\documentclass[12pt]{article}
\usepackage{graphicx,subfigure}
\usepackage{amsmath,amsthm,amssymb,enumerate}
\usepackage{euscript,mathrsfs}
\usepackage{color}
\usepackage{dsfont}
\usepackage{url}
\usepackage{bm}
\usepackage[left=2cm,right=2cm,top=3.5cm,bottom=3.5cm]{geometry}
\usepackage{color}
\usepackage[framemethod=tikz]{mdframed}
\usepackage{mathtools}
\allowdisplaybreaks

\newtheorem{Theorem}{Theorem}[section]
\newtheorem{Proposition}[Theorem]{Proposition}
\newtheorem{Lemma}[Theorem]{Lemma}
\newtheorem{Corollary}[Theorem]{Corollary}
\theoremstyle{definition}
\newtheorem{Definition}[Theorem]{Definition}
\newtheorem{Remark}[Theorem]{Remark}

\newcommand{\softd}{{\leavevmode\setbox1=\hbox{d}%
		\hbox to 1.05\wd1{d\kern-0.4ex{\char039}\hss}}}
\DeclarePairedDelimiter{\norm}{\lVert}{\rVert}

\newcommand{\ep}{\varepsilon}
\newcommand{\vme}{\vm_\ep}

\newcommand{\tvE}{\widetilde{E}}
\newcommand{\tvm}{\widetilde{\bm{m}}}
\newcommand{\bfphi}{\boldsymbol{\varphi}}
\newcommand{\Ov}[1]{\overline{#1}}
\newcommand{\vr}{\varrho}
\newcommand{\vre}{\vr_\ep}
\newcommand{\wtilde}{\widetilde}
\newcommand{\vue}{\vu_\ep}
\newcommand{\tvr}{\wtilde \vr}
\newcommand{\vt}{\vartheta}
\newcommand{\vu}{\bm{u}}
\newcommand{\vm}{\bm{m}}
\newcommand{\vn}{\bm{n}}

\newcommand{\Div}{{\rm div}_x}
\newcommand{\Grad}{\nabla_x}
\newcommand{\dx}{\,{\rm d} {x}}
\newcommand{\dt}{\,{\rm d} t }
\newcommand{\vU}{\bm{U}}
\newcommand{\intO}[1]{\int_{\Omega} #1 \ \dx}
\newcommand{\D}{{\rm d}}
\newcommand{\br}{ \notag \\ }

\newcommand{\jump}[1]{[[ #1 ]]}
\newcommand{\avs}[1]{\{\{ #1\}\}}

\mdfdefinestyle{MyFrame}{%
	linecolor=black,
	outerlinewidth=1pt,
	roundcorner=5pt,
	innertopmargin=\baselineskip,
	innerbottommargin=\baselineskip,
	innerrightmargin=10pt,
	innerleftmargin=10pt,
	backgroundcolor=white!20!white}

\begin{document}


\title{\bf Oscillatory approximations and maximum entropy principle
for the Euler system of gas dynamics}

\author{Eduard Feireisl
	\thanks{The work of E.F. was partially supported by the
		Czech Sciences Foundation (GA\v CR), Grant Agreement
		24--11034S. The Institute of Mathematics of the Academy of Sciences of
		the Czech Republic is supported by RVO:67985840.
		E.F. is a member of the Ne\v cas Center for Mathematical Modelling and Mercator Fellow in SPP 2410 ``Hyperbolic Balance Laws: Complexity, Scales and Randomness".}
	\and
M\' aria Luk\'a\v{c}ov\'a-Medvi\softd ov\'a
\thanks{The work of  M.L.-M. was supported by the Gutenberg Research College and by
		the Deutsche Forschungsgemeinschaft (DFG, German Research Foundation) -- project number 233630050 -- TRR 146 and
		project number 525800857 -- SPP 2410 ``Hyperbolic Balance Laws: Complexity, Scales and Randomness".
		She is also grateful to  the  Mainz Institute of Multiscale Modelling  for supporting her research. C.~Y. was supported by
the DFG project number 525800857 -- SPP 2410 ``Hyperbolic Balance Laws: Complexity, Scales and Randomness".}  
\and Changsheng Yu$^\dagger$}

\date{}

\maketitle

\centerline{$^*$Institute of Mathematics of the Academy of Sciences of the Czech Republic}
\centerline{\v Zitn\' a 25, CZ-115 67 Praha 1, Czech Republic}
\centerline{feireisl@math.cas.cz}

\medskip

\centerline{$^\dagger$Institute of Mathematics, Johannes Gutenberg-University Mainz}
\centerline{Staudingerweg 9, 55 128 Mainz, Germany}
\centerline{\{lukacova, {chayu}\}@uni-mainz.de}

\begin{abstract}
	
We show that the measure--valued solutions of the Euler system of gas dynamics generated by oscillatory sequences of consistent approximations violate the principle of
maximal entropy production formulated by Dafermos.	Numerical results illustrate that solutions obtained by standard numerical methods may be oscillatory and thus do not comply with the Dafermos criterion.

\end{abstract}


{\small

\noindent
{\bf 2020 Mathematics Subject Classification:}{ 35 Q 31
(primary); 35 D 30, 35 A 02
(secondary) }

\medbreak

\newpage
\tableofcontents

}
\section{Introduction}
\label{i}

Our goal is to show that the measure valued solutions of the Euler system of gas dynamics generated by oscillatory sequences of consistent approximations violate the principle of maximal entropy production in the sense of Dafermos
\cite{Dafer}. The result is at odds with a seemingly correct argument that oscillations enhance the mechanical energy dissipation thus increasing the entropy production. The main ingredients of our approach include:
\begin{enumerate}
	\item A proper identification of weak limits of consistent approximations --
	the dissipative measure-valued (DMV) solutions introduced in \cite{BreFei17A}, \cite{FeLMMiSh}.
	\item A rigorous verification of the statement that weakly converging sequences of consistent approximations cannot approach a weak solution of the Euler system, cf. \cite{MarEd}, \cite{FeiHof22}.
	\item
	A DMV solution of the Euler system that complies with Dafermos' admissibility
	criteria is necessarily a weak solution. Moreover, the same holds true in a smaller class of computable DMV solutions that can be identified with limits
	of consistent approximations.
\end{enumerate}	
The present paper addresses mostly the item 3 in the above list, together with some related problems concerning the amplitude of oscillatory solutions. A striking consequence
of our results is that the following so-far well accepted admissibility criteria
for the Euler system, namely:
\begin{itemize}
	\item admissible solutions should maximize the entropy production rate,
	see Dafermos \cite{Dafer}, DiPerna \cite{DiP2}, DeLellis and Sz\' ekelyhidi \cite{DelSze3}, but also negative examples concerning the isentropic Euler system
	by Chiodaroli and Kreml \cite{ChiKre}, Klein \cite{Klein2022}, Markfelder \cite{Mark2024},
	\item admissible solutions should be limits of consistent approximations (e.g. vanishing viscosity--dissipation limit, numerical approximations),
	see Constantin \cite{Const1}, Nussenzveig--Lopes et al \cite{NLSeWi},
\end{itemize}	
are \emph{not compatible} as long as the consistent approximations exhibit oscillations and/or concentrations.

\subsection{Euler system of gas dynamics}
\label{E}

The motion of a gas contained in a bounded domain $\Omega \subset R^d$, $d=1,2,3$ can be described by the
\emph{Euler system} of gas dynamics:
\begin{align}
	\partial_t \vr + \Div \vm &= 0, \label{E1}\\
	\partial_t \vm + \Div \bigg( \frac{\vm \otimes \vm}{\vr} \bigg) +
	\Grad p &= 0, \label{E2} \\
	\partial_t E + \Div \bigg[ \Big( E + p \Big) \frac{\vm}{\vr} \bigg] &= 0.
	\label{E3}	
\end{align}
Here $\vr = \vr(t,x)$ is the mass density, $\vm = \vm(t,x) = \vr \vu (t,x)$ the momentum with the velocity $\vu$,
$p$ the pressure, and $E$ the energy. In addition, we impose the impermeability condition
\begin{equation} \label{E4}
	\vm \cdot \bm{n}|_{\partial \Omega} = 0,\ \bm{n} \ \mbox{the outer normal vector to}\ \partial \Omega.
\end{equation}

The energy associated to the system takes the form
\[
E = \frac{1}{2} \frac{|\vm|^2}{\vr} + \vr e,
\]
where the internal energy $e$ as well as the pressure $p$ are determined constitutively by means of equations of state (EOS).
We suppose that $e$ and $p$ are interrelated through
\emph{Gibbs' law}
\begin{equation} \label{E5}
	\vt D s = De + p D \bigg( \frac{1}{\vr} \bigg),
\end{equation}	
where $s$ is a new state variable -- the entropy. Moreover, we assume \emph{thermodynamic stability hypothesis},
meaning the energy
\[
E = \frac{1}{2} \frac{|\vm|^2}{\vr} + \vr e (\vr, S),\
S = \vr s
\]
is a strictly convex function of the conservative-entropy variables $(\vr, \vm, S)$, cf. \cite[Chapter 4, Section 4.1.6]{FeLMMiSh}.

As is well known, the Euler system inherits the difficulties pertinent to general systems of non-linear conservation laws. In particular, smooth solutions
emanating from smooth initial data exist only on a possibly short time interval, whereas singularities in the form of shock waves appear in a finite time.
The shock type singularities can be accommodated by considering a larger class of weak solutions, however, the prize to pay is the lack of uniqueness,
cf. the standard reference monographs by Benzoni--Gavage, Serre \cite{BenSer}, Dafermos \cite{D4a}, or Smoller \cite{SMO}.

To recover well posedness,
extra \emph{admissibility conditions} inspired by the physical background of gas dynamics can be appended to the weak formulation of the problem. Notably,
the Second law of thermodynamics can be incorporated in the form of \emph{entropy inequality}
\begin{equation} \label{E8a}
	\partial_t S + \Div \bigg( S \frac{\vm}{\vr} \bigg)  \geq 0.
\end{equation}
In view of the impermeability boundary condition \eqref{E4}, the total entropy
\[
t \mapsto \intO{ S(t, \cdot) }
\]
is a non-decreasing function of time. Unfortunately, the recent results based on the method of convex integration show that the Euler
system, even if appended by the entropy inequality \eqref{E8a}, remains ill posed at least in the multidimensional case $d=2,3$,
see \cite{ChFe2023}, \cite{FeKlKrMa}, Klingenberg et al. \cite{KlKrMaMa}.

\subsection{Dafermos' admissibility criterion}

In his seminal work \cite{Dafer}, Dafermos  proposed a refined admissibility criterion based on the order relation $\prec_{\rm D}$
between two (weak) solutions of the Euler system $(\vr_i, \vm_i, S_i)$, $i=1,2$:
\begin{align}
	(\vr_1, \vm_1, S_1) &\prec_{\rm D}
	(\vr_2, \vm_2, S_2)\br \ &\Leftrightarrow_{\rm{def}} \br
	\mbox{there exists}\ \tau \geq 0 \ \mbox{such that}\ (\vr_1, \vm_1, S_1)(t, \cdot) &=
	(\vr_2, \vm_2, S_2)(t, \cdot) \ \mbox{for all}\
	t \leq \tau, \br
	\frac{\D^+}{\dt} \intO{ S_2 (\tau, \cdot) } &>
	\frac{\D^+}{\dt}  \intO{ S_1 (\tau, \cdot) }.
	\label{Df1}
\end{align}
Here
\[
\frac{\D^+}{\dt} F(t) = \limsup_{\delta \to 0+} \frac{ F(t + \delta) - F(t) }{\delta}
\]
is the right-derivative. In other words, the solution $(\vr_2, \vm_2, S_2)$ produces (locally) more entropy than
$(\vr_1, \vm_1, S_1)$.

We say that a solution $(\tvr, \tvm, \widetilde{S})$ of the Euler system is \emph{maximal dissipative} (or Dafermos admissible) if it is maximal with respect to the order relation
$\prec_{\rm D}$. Maximal should be understood in the sense there is no other solution $(\vr, \vm, {S})$ such that
\[
(\tvr, \tvm, \widetilde{S}) \prec_{\rm D} (\vr, \vm, {S}),
\]
which coincides with the original definition proposed in \cite{Dafer}.

\subsection{Dafermos criterion in the class of measure-valued solutions}

The concept of \emph{measure-valued} solution has been proposed by DiPerna \cite{DiP2}
and later developed in a series of papers by DiPerna and Majda \cite{DiPMaj88}, \cite{DiPMaj87}, \cite{DiPMaj87a} to capture possible oscillations
in families of \emph{consistent approximations} of the Euler system. Examples of consistent approximations are the vanishing viscosity limit or
limits of a suitable numerical scheme, see the \cite[Chapter 5]{FeLMMiSh}. Solutions are represented by families of parametrized (Young) measures yielding the probability the approximate
sequence occupies a specific part of the phase space.

A proper measure-valued formulation of the (full compressible) Euler system \eqref{E1}--\eqref{E3} taking into account both oscillations and concentrations
in approximate sequences was introduced in \cite{BreFei17}, see Section \ref{d} below. Dafermos' criterion can be easily reformulated in terms
of the expected values of the Young measure. Specifically, we set
\begin{equation} \label{ev} 
	\vr(t,x) = \big< \mathcal{V}_{t,x}, \tvr \big>, \ \vm(t,x) = \big< \mathcal{V}_{t,x}, \tvm \big>,\ S(t,x) = \Big< \mathcal{V}_{t,x}, \widetilde{S} \Big>,
\end{equation}
where $\{ \mathcal{V}_{t,x} \}$ is the Young measure associated to the measure-valued solution, and retain the original
Dafermos formulation \eqref{Df1} in the same form. This approach is in line with the general idea that any measure-valued solution should be identified with a set of observables -- expected values of physically relevant quantities, while the specific form of the Young measure is irrelevant, see \cite{FeLMScSh}.

Our first result, stated rigorously in Theorem \ref{dT1} below, asserts that
if a measure-valued solution is maximal dissipative, then the associated Young measure reduces to the Dirac mass. In particular, any
consistent approximation converges strongly (in the $L^p$-sense), and its limit is a weak solution of the Euler system. The second result
is the extension of the former one to the narrower class of computable solutions generated by a family of suitable consistent approximations.
In particular, we show in Theorem \ref{cadT} that any limit of oscillatory consistent approximations cannot be maximal dissipative in the
class of computable solutions. In Section~\ref{NUM} we document these results by a series of numerical experiments.

\subsection{Measure-valued solutions with small energy defect}

Solutions of the Euler system \eqref{E1}--\eqref{E3}, with the impermeability boundary condition \eqref{E4}
and the initial condition
\[
\vr(0, \cdot) = \vr_0,\ \vm(0, \cdot) = \vm_0, \ S(0, \cdot) = S_0,
\]
conserve the total energy:
\[
\intO{ E(\vr, \vm, S)(t, \cdot)} = \intO{ E(\vr_0, \vm_0, S_0) }\ \mbox{for any}\ t > 0.
\]
This property is preserved even in the class of weak solutions.
The measure-valued solutions, however, experience a non-negative energy defect
\[
D_{\rm E}(t) = \intO{ E(\vr_0, \vm_0, S_0) } - \intO{ E(\vr, \vm, S)(t, \cdot)} \geq 0,
\]
see Section \ref{cp}. As we shall see below, a measure valued solution coincides with a weak solution if and only if the energy defect vanishes.
Intuitively, it is desirable to consider measure-valued solutions with small energy defect. As shown in Theorem \ref{Tdef}, for any $\delta > 0$, there exists
a measure-valued solution with the energy defect smaller than $\delta$.

\subsection{Energy preserving measure-valued solutions}

Our choice of state variables specified in \eqref{ev} is related to the system
of field equations available for the measure-valued solutions constructed
in \cite{BreFei17}. In the context of the complete Euler system \eqref{E1}--\eqref{E3}, however,
a ``natural'' choice are the conservative variables
\begin{equation} \label{ev1}
	\vr(t,x) =\big< \mathcal{V}_{t,x}, \tvr \big>, \ \vm(t,x) = \big< \mathcal{V}_{t,x}, \tvm \big>,\ E(t,x) = \Big< \mathcal{V}_{t,x}, \widetilde{E} \Big>, 
\end{equation}
where $E$ is the total energy of the system. Note, however, there are no sufficient {\it a priori} bounds on the convective terms in \eqref{E3} that would
justify the Young measure formulation of energy equation \eqref{E3}.

In order to show robustness of our results, we follow the strategy of
Fjordholm, Mishra, and Tadmor \cite{FjMiTa1} based on the
(hypothetical) existence of measure-valued solutions generated by uniformly bounded sequences of approximate solutions. This gives rise
to the truly measure-valued formulation of the Euler system based on a single Young measure sitting on the conservative variables
$(\tvr, \tvm, \tvE)$. In particular, the principle of energy conservation pertinent to the Euler system is automatically satisfied.
We call this class of solutions energy preserving measure-valued solutions.

We show that our main result stated originally in Theorem \ref{dT1} remains
valid in the class of energy preserving measure-valued solutions. Specifically, if an energy preserving measure-valued solution complies with Dafermos' maximal entropy admissibility principle, then it is a weak solution of the Euler system, see Theorem \ref{TB1}.

The paper is organized as follows. In Section \ref{d}, we recall the concept of dissipative measure-valued solution for the Euler system introduced
in \cite{BreFei17}. Section \ref{M} contains the main result on regularity of maximal dissipative solutions. Section \ref{ca} is devoted to consistent approximations of the Euler system. The proofs of the main results concerning admissibility are completed in Section \ref{vd}. The existence of
solutions with small energy defect is discussed in Section \ref{def}.
The extension of the results to the class of energy preserving measure-valued solutions is presented in Section \ref{B}.
Section~\ref{NUM} illustrates theoretical results. We demonstrate that if a solution exhibits oscillations or concentrations then standard numerical methods  are not compatible with the Dafermos maximal entropy production rate criterion. We conclude the paper by formulating a hypothesis on the existence of a maximal computable measure-valued solution and its approximation by numerical methods.

\section{Dissipative measure-valued (DMV) solutions}
\label{d}

For the sake of simplicity,
we consider a polytropic EOS in the form
\begin{equation} \label{r1a}
	p = (\gamma - 1) \vr e ,\ \gamma > 1.
\end{equation}
Moreover, without loss of generality, we introduce the \emph{absolute temperature} $\vt$ proportional to $e$,
\begin{equation} \label{r2a}
	e = c_v \vt, \ c_v = \frac{1}{\gamma -1}.
\end{equation}
This yields
\begin{equation}\label{r3a}
	p(\vr, S) = (\gamma - 1) \vr e (\vr, S) = 
	\begin{cases}
		\vr^\gamma \exp \Big( \frac{S}{c_v \vr} \Big) \ &\mbox{if} \ \vr > 0, \\
		0 \ &\mbox{if}\ \vr = 0,\ S \leq 0, \\
		\infty \ &\mbox{if}\ \vr = 0,\ S > 0.
	\end{cases}
\end{equation}
In particular,

\textcolor{black}{\begin{equation} \label{r1}
		\frac{\partial (\vr e(\vr, S))}{\partial S} = \vt > 0,
\end{equation}}
meaning $\vr e(\vr, S)$ is a strictly increasing function of $S$.
Introducing the kinetic energy
\begin{equation} \label{RR1}
	\frac{1}{2} \frac{|\vm|^2}{\vr} =
	\begin{cases}
		\frac{1}{2} \frac{|\vm |^2}{\vr} \ &\mbox{if}\ \vr > 0,\\
		0 \ &\mbox{if}\ \vr = 0, \vm = 0, \\
		\infty \ &\mbox{otherwise}
	\end{cases}
\end{equation}
we can check that the total energy
\[
E(\vr, S) = \frac{1}{2} \frac{|\vm|^2}{\vr} + \vr e(\vr, S)
\]
is a convex l.s.c. function of the conservative-entropy variables $(\vm, \vr, S) \in R^{d+2}$, strictly convex on its domain,
in accordance with hypothesis of thermodynamic stability, cf. \cite[Chapter 2, Section 2.2.4]{FeLMMiSh}.
As a matter of fact, all results claimed in the present paper apply to a much larger class of EOS as long as Gibbs' relation and hypothesis of thermodynamic stability remain valid.

\subsection{DMV solutions}

A dissipative measure-valued solution consists of a parametrized (Young) measure
\begin{equation} \label{d1}
	\mathcal{V} = \mathcal{V}_{t,x}: (t,x) \in (0,T) \times \Omega \to \mathfrak{P} (R^{d+2}),\
	\mathcal{V} \in L^\infty_{\rm weak-(*)}((0,T) \times \Omega; \mathfrak{P} (R^{d+2})),
\end{equation}
where $\mathfrak{P}$ denotes the set of Borel probability measures defined
on the space of ``dummy'' variables
\[
R^{d+2} = \Big\{ (\tvr, \tvm, \widetilde{S}) \mid  (\tvr, \tvm,  \widetilde{S}) \in R^{d+2} \Big\},
\]
and a concentration measure
\begin{equation} \label{d1a}
	\mathfrak{C} \in L^\infty_{\rm weak-(*)}(0,T; \mathcal{M}^+ (\Ov{\Omega}; R^{d \times d}_{\rm sym})),
\end{equation}
where
the symbol $\mathcal{M}^+ (\Ov{\Omega}; R^{d \times d}_{\rm sym}))$ denotes the set of all matrix-valued positively semi-definite finite Borel measures on
the compact set $\Ov{\Omega}$.

\begin{Definition}[{\bf DMV solution}] \label{Dd2}
	
	We shall say that a parametrized measure $\{ \mathcal{V}_{t,x} \}$ and
	a concentration measure $\mathfrak{C}$ represent \emph{dissipative measure-valued (DMV) solution} of problem \eqref{E1}--\eqref{E3}, \eqref{E4}, with the initial data $(\vr_0, \vm_0, S_0)$ in $(0,T) \times \Omega$ if the following holds:
	\begin{itemize}		
		\item
		The equation of continuity
		\begin{equation} \label{d2}
			\int_0^T \intO{ \Big[ \big< \mathcal{V}_{t,x}; \tvr \big> \partial_t \varphi(t,x) + \big< \mathcal{V}_{t,x}; \widetilde{\vm} \big> \cdot \Grad \varphi(t,x) \Big]} \dt = - \intO{ \vr_0(x) \varphi(0, x) } 	
		\end{equation}
		holds for any $\varphi \in C^1([0,T) \times \Ov{\Omega})$.
		\item
		The momentum equation
		\begin{align}
			\int_0^T &\intO{ \Bigg[ \big< \mathcal{V}_{t,x}; \widetilde{\vm} \big> \cdot \partial_t \bfphi(t,x) + \bigg< \mathcal{V}_{t,x}; \mathds{1}_{\tvr > 0} \bigg(
				\frac{ \tvm \otimes \tvm}{\tvr} + p(\tvr, \widetilde{S}) \mathbb{I} \bigg)  \bigg>  \Grad \bfphi(t,x) 	 \Bigg] } \dt \br
			&= - \int_0^T \int_{\Ov{\Omega}} \Grad \bfphi(t,x) : \D \mathfrak{C} (t,x)
			\label{d3}	- \intO{ \vm_0(x) \cdot \bfphi(0, x) }
		\end{align}
		holds for any $\bfphi \in C^1([0,T) \times \Ov{\Omega}; R^d)$, $\bfphi \cdot \bm{n}|_{\partial \Omega} = 0$.
		\item
		The entropy inequality
		\begin{align}
			\int_0^T &\intO{ \bigg[ \Big< \mathcal{V}_{t,x}; \widetilde{S} \Big> \partial_t \varphi(t,x) + \bigg< \mathcal{V}_{t,x}; \mathds{1}_{\tvr > 0} \widetilde{S} \frac{\tvm}{\tvr} \bigg> \cdot \Grad \varphi(t,x) \bigg] } \dt
			\br &\leq - \intO{ S_0(x) \varphi (0, x) }
			\label{d4}		
		\end{align}
		holds for any $\varphi \in C^1([0,T) \times \Ov{\Omega})$, $\varphi \geq 0$.
		\item
		The energy compatibility condition
		\begin{align}
			\intO{ E(\vr_0, \vm_0, S_0) }
			&\geq \intO{ \Big< \mathcal{V}_{t,x}; E(\tvr, \tvm, \widetilde{S})\Big>}	+   r(d, \gamma)\int_{\Ov{\Omega}} \ \D {\rm trace}[ \mathfrak{C}](t,x) ,\br
			r(d, \gamma) &= \min \Bigg\{ \frac{1}{2}; \frac{d \gamma}{\gamma - 1} \Bigg\}
			\label{d5}
		\end{align}	
		holds for a.a. $t \in (0,T)$.
		
	\end{itemize}
	
\end{Definition}

The concept of DMV solution was introduced in \cite{BreFei17} and later elaborated in the monograph \cite[Chapter 5]{FeLMMiSh}. Kr\" oner and Zajaczkowski \cite{KrZa}
introduced a concept of measure-valued solution to the Euler system postulating entropy \emph{equation} rather than inequality \eqref{d4}. Unfortunately,
the physically admissible shock waves weak solutions do not fit in this class.

\subsection{Basic properties of DMV solutions}

We start with global existence of DMV solutions proved in \cite[Proposition 3.8]{BreFeiHof19C}.

\begin{Proposition}[{\bf Global existence}] \label{Pd1}
	Let $\Omega \subset R^d$, $d=2,3$ be a bounded domain with smooth boundary. Suppose the thermodynamic functions $p$, $e$, and $s$ satisfy hypotheses \eqref{r1a}--\eqref{r3a}. Let the initial data belong to the class
	\begin{align}
		(\vr_0, \vm_0, S_0) \in L^1(\Omega; R^{d+2}),\
		\intO{ E(\vr_0, \vm_0, S_0) } < \infty, \ S_0 \geq \underline{s} \vr_0 	\
		\mbox{a.a. in}\ \Omega,
		\label{dclass}
	\end{align}
	where $\underline{s} \in R$ is a given constant.	
	
	Then the Euler system \eqref{E1}--\eqref{E3}, \eqref{E4} admits a DMV solution
	in $(0,T) \times \Omega$ specified in Definition \ref{Dd2}. In addition,
	the solution satisfies
	\begin{equation}\label{dd1}
		\mathcal{V}_{t,x} \Big\{ \tvr \geq 0,\ \widetilde{S} \geq \underline{s} \tvr \Big\} = 1 \ \mbox{for a.a.}\ (t,x) \in (0,T) \times \Omega.
	\end{equation}	
\end{Proposition}	

\begin{Remark} \label{Rr1}
	
	Hypothesis $S_0 \geq \underline{s} \vr_0$ and its counterpart in \eqref{dd1} correspond to the entropy minimum principle valid in the
	class of DMV solutions, see e.g. \cite{BreFeiHof19C} or \cite[Chapter 4, Section 4.1.5]{FeLMMiSh}. The lower bound on the entropy yields uniform {\it a priori bounds}
	\[
	\vr, S \in L^\infty(0,T; L^\gamma (\Omega)),\ \vm \in L^\infty(0,T; L^{\frac{2 \gamma}{\gamma + 1}}(\Omega; R^d)),
	\]
	necessary for the existence theory.
	
\end{Remark}

\begin{Remark} \label{Rr2}
	
	Smoothness of the boundary $\partial \Omega$ is necessary for the artificial ``viscosity'' construction used in the existence proof in
	\cite[Proposition 3.8]{BreFeiHof19C}. It can be possibly relaxed by means of a suitable approximation of $\Omega$ by a family of smooth invading
	domains.
\end{Remark}

In addition, the DMV solutions enjoy other interesting properties, some of them listed below:

\begin{itemize}
	\item {\bf Convexity, compactness.} The set $\mathcal{U}[\vr_0, \vm_0, S_0]$ of all DMV solutions in $(0,T) \times \Omega$ emanating from given initial data $(\vr_0, \vm_0, S_0)$ is
	convex and compact in suitable (weak) topologies, see \cite[Chapter 5, Theorem 5.2]{FeLMMiSh} for details. 	
	\item {\bf Compatibility.}
	If $(\vr, \vm, S)$ are continuously differentiable and $\vr$ bounded below away from zero, then
	the DMV solution is a classical solution, specifically,
	\begin{equation} \label{sm}
		\mathcal{V}_{t,x} = \delta_{[\vr(t,x), \vm(t,x), S(t,x)]} ,\ \mathfrak{C} = 0,
	\end{equation}
	and $(\vr, \vm, S)$ satisfy \eqref{E1}--\eqref{E3}, see \cite[Chapter 5, Theorem 5.7]{FeLMMiSh}.
	
	\item{\bf Weak--strong uniqueness.}
	A DMV solution coincides with the strong solution (in the sense specified in \eqref{sm}) emanating from the same initial data as long as
	the latter exists, see \cite[Chapter 6, Theorem 6.2]{FeLMMiSh}. Note that \emph{local} existence of strong solutions for smooth initial data was proved by
	Schochet \cite{SCHO1}.
	
\end{itemize}		

\section{Main result on regularity of maximal DMV solutions}
\label{M}

Our main result concerning regularity of DMV solutions will be formulated in terms of the expected values introduced in \eqref{ev}.
Note that the triple $(\vr, \vm, S)$ represents the so-called
dissipative solution introduced in \cite[Chapter 5]{FeLMMiSh}. The same concept
has also been used in \cite{BreFeiHof19C} to identify a semigroup selection for the Euler system.

Under the hypotheses of Proposition \ref{Pd1}, the expected values can be uniquely extended as functions defined for any time $t \in [0,T]$. More precisely, we set
\begin{align*}
	& \int_\Omega \vr(t, x) \varphi(x) \mbox{d} x  =
	\lim_{\delta \to 0+} \frac{1}{\delta} \int_{t-\delta}^t \int_\Omega \big< \mathcal{V}_{t,x};\tvr \big> \varphi(x) \mbox{d} x \mbox{d} s, \\
	& \int_\Omega \vm(t, x) \bfphi(x) \mbox{d} x  =
	\lim_{\delta \to 0+} \frac{1}{\delta} \int_{t-\delta}^t \int_\Omega \big< \mathcal{V}_{t,x};\tvm \big> \bfphi(x) \mbox{d} x \mbox{d} s, \\
	& \int_\Omega S(t, x) \varphi(x) \mbox{d} x  =
	\lim_{\delta \to 0+} \frac{1}{\delta} \int_{t-\delta}^t \int_\Omega \Big< \mathcal{V}_{t,x}; \widetilde S\Big> \varphi(x) \mbox{d} x \mbox{d} s.
\end{align*}
Specifically,
\begin{align}
	\vr &\in C_{\rm weak}([0,T]; L^\gamma (\Omega)), \ \vr(0, \cdot) = \vr_0, \br
	\vm &\in C_{\rm weak}([0,T]; L^{\frac{2 \gamma}{\gamma + 1}}(\Omega; R^d)),\
	\vm(0, \cdot) = \vm_0, \br
	S &\in BV_{\rm weak}(0,T; L^\gamma (\Omega)),\
	S_0 \leq S(0+, \cdot),\ S(t-, \cdot) \leq S(t+, \cdot) \ \mbox{for any}\ t \in (0,T).
	\label{dd2}
\end{align}	
see \cite{BreFeiHof19C} for details.
Consequently, setting $S(0, \cdot) = S_0$, the expected value of the entropy can be uniquely defined as a (weakly) c\` agl\` ad function of $t \in [0,T]$ with values in $L^\gamma(\Omega)$, (weakly) continuous at $t = 0$. Finally, it follows from the entropy inequality \eqref{d4} that the total entropy
\[
t \in [0,T] \to \intO{ S(t, \cdot) }
\]
is a non-decreasing (c\` agl\` ad) function. In particular,
\begin{align}
	0 \leq \frac{\D^+}{\dt} \intO{ S(\tau, \cdot) } &\leq \infty
	\ \mbox{for any}\ \tau \in [0,T),\br
	\frac{\D^+}{\dt} \intO{ S(t, \cdot) } &= \frac{\D }{\dt} \intO{S(t, \cdot)} < \infty \ \mbox{for a.a.}\ t \in (0,T).
	\label{dd3}
\end{align}	

We conclude that the order relation $\prec_{\rm D}$ introduced in \eqref{Df1}
can be extended to DMV solutions in terms of the expected values of the
phase variables $(\vr, \vm, S)$ introduced above. This observation justifies the following definition.

\begin{Definition}[{\bf Maximal DMV solution}] \label{Dd1}	
	A DMV solution $\{ \mathcal{V}_{t,x} \}, \mathfrak{C}$ is called \emph{maximal} if the associated trio of
	expected values $(\vr, \vm, S)$ is maximal with respect to the order relation
	$\prec_{\rm D}$ specified in \eqref{Df1}.
\end{Definition}

\textcolor{black}{While the global existence of  DMV solutions is known, cf.~Proposition~\ref{Pd1}, the existence of \emph{maximal DMV} is an open problem.
	As a matter of fact any maximal DMV solution is necessarily a weak solution as stated in our main result below.}

\begin{mdframed}[style=MyFrame]
	\begin{Theorem}[{\bf Regularity of maximal DMV solutions}] \label{dT1}
		Under the hypotheses of Proposition \ref{Pd1}, let $\{ \mathcal{V}_{t,x} \}$,
		$\mathfrak{C}$ be a maximal DMV solution of the Euler system in $(0,T) \times \Omega$.
		
		Then
		\[
		\mathcal{V}_{t,x} = \delta_{[\vr(t,x), \vm(t,x), S(t,x)]},\ \mathfrak{C} = 0,
		\] 	
		and $(\vr, \vm, S)$ satisfy the following (generalized) Euler system:
		\begin{align}
			\partial_t \vr + \Div \vm &= 0,\ \vr(0, \cdot) = \vr_0, \label{E6}\\
			\partial_t \vm + \Div \bigg( \frac{\vm \otimes \vm}{\vr} \bigg) +
			\Grad p(\vr, S) &= 0,\ \vm(0, \cdot) = \vm_0, \label{E7} \\
			\partial_t S + \Div \bigg( S \frac{\vm}{\vr} \bigg) & \geq 0,\ S(0, \cdot) = S_0,
			\label{E8}	\\
			\frac{\D }{\dt} \intO{ E(\vr, \vm, S) } &=0,\br \intO{E(\vr, \vm, S)(0+, \cdot)} &= \intO{ E(\vr_0, \vm_0, S_0)}, \label{E9}
		\end{align}
		in the sense of distributions.
		
	\end{Theorem}		
	
\end{mdframed}

The system \eqref{E6}--\eqref{E9} can be viewed as a weak formulation of the Euler system \eqref{E1}--\eqref{E3}
in the total absence of {\it a priori} bounds to control the convective terms in the energy equation \eqref{E3}.
A similar approach has already been used in the context of the compressible Navier--Stokes--Fourier system in \cite{FeiNovOpen}.
The proof of Theorem \ref{dT1} will be performed in Section \ref{vd} below. The following corollary will follow immediately
from the proof.

\begin{Corollary} \label{dC1}
	
	Suppose that $\Omega \subset R^d$ as well as the initial data $(\vr_0, \vm_0, S_0)$ satisfy the hypotheses of Proposition \ref{Pd1}.
	
	Then either the (generalized) Euler system \eqref{E6}--\eqref{E9} admits a unique (weak) solution in $(0,T) \times \Omega$, or there
	is an infinite family $\{ \mathcal{V}^n_{t,x} \}$, $\mathfrak{C}^n$ of DMV solutions such that the associated expected values
	$(\vr^n, \vm^n, S^n)$ satisfy
	\[
	(\vr^n, \vm^n, S^n) \prec_{\rm D} (\vr^{n+1}, \vm^{n+1}, S^{n+1}) ,\ n = 1,2,\dots
	\]	
	
\end{Corollary}

\section{Consistent approximations, main result on convergence}
\label{ca}

The concept of consistent approximation originates in numerical analysis.
Consistent approximations satisfy the Euler system in a weak sense modulo an approximation error, see \cite[Chapter 5]{FeLMMiSh}.
Some of the DMV solutions introduced above may be generated by families of consistent approximations.

\begin{Definition}[{\bf Consistent approximation}] \label{Dca1}	
	
	We say that a family $(\vre, \vme, {S}_\ep)_{\ep > 0}$ is \emph{consistent approximation} of
	the Euler system \eqref{E1}--\eqref{E3}, \eqref{E4} with the initial data $(\vr_0, \vm_0, S_0)$
	if holds:
	\begin{itemize}
		
		\item Approximate total energy inequality
		\begin{equation} \label{ca2}
			\intO{ E(\vre, \vme, S_\ep)(\tau, \cdot) } \leq \intO{ E(\vr_{0}, \vm_{0} , S_{0} ) } + e^1_\ep
		\end{equation}
		for a.a. $\tau \in (0,T)$, where $e^1_\ep \to 0$ as $\ep \to 0$.
		\item
		Approximate equation of continuity
		\begin{align}
			\bigg[ \intO{ \vre (t, \cdot)\varphi  } \bigg]_{t = \tau_1-}^{t= \tau_2-} &= \int_{\tau_1}^{\tau_2} \intO{  \vme \cdot \Grad \varphi  } \dt + \int_{[\tau_1, \tau_2)} \D e^2_{\ep}[\varphi],\br \vre(0-, \cdot) &\equiv \vr_{0} \br
			e^2_\ep[\varphi] &\in \mathcal{M}[0,T],\ \int_{[0,T]} \D |e^2_\ep [\varphi] | \to 0 \ \mbox{as}\ \ep \to 0
			\label{ca3}
		\end{align}	
		holds for any $0 \leq \tau_1 \leq \tau_2 \leq T$, and any $\varphi \in C^\infty(\Ov{\Omega})$.
		
		\item Approximate momentum balance
		\begin{align}
			\bigg[ \intO{ \vme (t, \cdot) \cdot \bfphi } \bigg]_{t = \tau_1-}^{t = \tau_2 -} &= 	
			\int_{\tau_1}^{\tau_2} \intO{ \mathds{1}_{\vre > 0} \bigg( \frac{\vme \otimes \vme }{\vre} + p(\vre, S_\ep) \mathbb{I} \bigg) :
				\Grad \bfphi  } \dt\br  &+ \int_{[\tau_1,\tau_2)} \D {e}^3_\ep [\bfphi],\br 
			\vme(0-, \cdot) &\equiv \vm_{0},
			\br
			{e}^3_\ep [\bfphi] &\in \mathcal{M}[0,T],\ \int_{[0,T]} \D |{e}^3_\ep [\bfphi] | \to 0 \ \mbox{as}\ \ep \to 0	
			\label{ca4}
		\end{align}
		holds for any $0 \leq \tau_1 \leq \tau_2 \leq T$,	$\bfphi \in C^\infty (\Ov{\Omega}; R^d)$, $\bfphi \cdot \bm{n}|_{\partial \Omega} = 0$.
		\item
		The entropy complies with the minimum principle
		\begin{equation} \label{mini}
			S_\ep \geq \underline{s} \vre \ \mbox{a.a. in}\ (0,T) \times \Omega,
		\end{equation}
		and approximate entropy inequality
		\begin{align}
			\bigg[ \intO{ S_\ep (t, \cdot) \varphi } \bigg]_{t = \tau_1-}^{t = \tau_2-} &\geq 	
			\int_{\tau_1}^{\tau_2} \intO{ \mathds{1}_{\vre > 0} S_\ep \frac{\vme}{\vre} : \Grad \varphi } \dt	
			+ \int_{[\tau_1, \tau_2)} \D e^4_\ep[\varphi],\br
			S_\ep(0-, \cdot)  &\equiv S_{0},   \br
			e^4_\ep[\varphi] &\in \mathcal{M}[0,T],\ \int_{[0,T]} \D |e^4_\ep [\varphi] | \to 0 \ \mbox{as}\ \ep \to 0
			\label{ca5}
		\end{align}
		holds for any $0 \leq \tau_1 \leq \tau_2 \leq T$, and any $\varphi \in C^\infty(\Ov{\Omega})$, $\varphi \geq 0$.
	\end{itemize}
	
\end{Definition}

Here, the \emph{consistency errors} $e^i_\ep[\varphi]$, $i = 2,3,4$ are measures on the compact interval $[0,T]$.
Accordingly, the quantities
\[
t \mapsto \intO{\vre (t, \cdot) \varphi},\ t \mapsto \intO{\vme(t, \cdot) \cdot \bfphi },\
t \mapsto \intO{ S_\ep (t, \cdot) \varphi }
\]
can be interpreted as BV functions defined in $[0,T]$.

Consistent approximations are usually represented by solutions of suitable numerical schemes, several examples can be found in the monograph
\cite[Part III]{FeLMMiSh}. Note, however, that
the definition of consistent approximation in \cite[Chapter 5]{FeLMMiSh} is slightly different, in fact more general, than Definition \ref{Dca1}.
\textcolor{black}{We work with Definition~\ref{Dca1} just to simplify the presentation, the results presented here can be proved also in the setting of \cite{FeLMMiSh}.} 
Another example is the vanishing viscosity approximation used in \cite{BreFeiHof19C} to construct DMV solutions for arbitrary initial data,
cf. Proposition \ref{Pd1}.

It is not difficult to see that any consistent approximation generates, up to a suitable subsequence, a DMV solution $\{ {\mathcal V}_{t,x} \}$, $\mathfrak{C}$ of the
Euler system, see \cite[Chapter 7]{FeLMMiSh}. Indeed it follows from the energy inequality \eqref{ca2} combined with the minimum entropy
principle \eqref{mini} that
\begin{align}
	\vre &\to \vr \ \mbox{weakly-(*) in}\ L^\infty (0,T; L^\gamma(\Omega)), \br
	\vme &\to \vm \ \mbox{weakly-(*) in}\ L^\infty (0,T; L^{\frac{2 \gamma}{\gamma + 1}}(\Omega; R^d)), \br
	S_\ep &\to S \ \mbox{weakly-(*) in}\ L^\infty (0,T; L^\gamma(\Omega))
	\label{gen1}
\end{align}
at least for a suitable subsequence.

Moreover, by the same token,
\begin{align}
	\mathds{1}_{\vre > 0} \bigg( \frac{\vme \otimes \vme }{\vre} + p(\vre, S_\ep) \mathbb{I} \bigg) \to
	\Ov{ \frac{\vm \otimes \vm }{\vr} + p(\vr, S) \mathbb{I} } \br \mbox{weakly-(*) in}\
	L^\infty (0,T; \mathcal{M}^+ (\Ov{\Omega}; R^{d \times d}_{\rm sym})).\notag
\end{align}
We set
\begin{equation} \label{gen2}
	\mathfrak{C} = \Ov{ \frac{\vm \otimes \vm }{\vr} + p(\vr, S) \mathbb{I} } - \bigg< \mathcal{V};
	\frac{\tvm \otimes \tvm }{\tvr} + p(\tvr, \widetilde{S}) \mathbb{I} \bigg>,
\end{equation}
see \cite[Chapter 7]{FeLMMiSh} for details. Now, it is a routine matter to perform the limits in \eqref{ca2}--\eqref{ca5} to conclude that
$\{ \mathcal{V}_{t,x} \}$, $\mathfrak{C}$ is a DMV solution \emph{generated} by the consistent approximation $(\vre, \vme, S_\ep)_{\ep > 0}$.

The following result asserts that the DMV solutions generated by \emph{weakly} convergent consistent approximations cannot be a weak solutions, even in the more general
sense \eqref{E6}--\eqref{E9}.

\begin{Proposition} \label{Pca1}
	Let $(\vre, \vme, S_\ep)_{\ep > 0}$ be a consistent approximation of the Euler system generating a Young measure $\mathcal{V}_{t,x}$ and a concentration
	measure $\mathfrak{C}$ -- a DMV solution of the Euler system in $(0,T) \times \Omega$. In addition, suppose that
	the expected values
	\[
	\vr(t,x) = \big< \mathcal{V}_{t,x}, \tvr \big>, \quad \vm(t,x) = \big< \mathcal{V}_{t,x}, \tvm \big>,\quad S(t,x) = \Big< \mathcal{V}_{t,x}, \widetilde{S} \Big>
	\]
	represent a (generalized) weak solution satisfying \eqref{E6}--\eqref{E9}.
	
	Then
	\[
	\mathcal{V}_{t,x} = \delta_{[\vr(t,x), \vm(t,x), S(t,x)]},\ \mathfrak{C} = 0,
	\]
	in particular the convergence of the consistent approximations is strong and a.a. (up to a subsequence).
\end{Proposition}

\noindent
For the proof see \cite[Chapter 7, Theorem 7.5]{FeLMMiSh}.

\begin{Definition}[{\bf Oscillatory consistent approximations}] \label{Dac2}
	
	Let $(\vre, \vme, S_\ep)_{\ep > 0}$ be a consistent approximation of the Euler system in $(0,T) \times \Omega$.
	We say the approximation is \emph{oscillatory} if it generates a DMV solution such that at least one of the following statements holds:
	\[
	\{ \mathcal{V}_{t,x} \}  \ne \{ \delta_{[\vr(t,x), \vm(t,x), S(t,x)]} \} \ \mbox{or} \
	\mathfrak{C} \ne 0.
	\]	
	
\end{Definition}

It follows from Proposition \ref{Pca1} that DMV solutions generated by oscillatory consistent approximations cannot be weak solutions
of the Euler system. This fact, combined Theorem \ref{dT1}, implies that limits of oscillatory consistent approximations cannot be maximal DMV solutions, meaning
they violate Dafermos' maximal dissipation criterion specified in Definition \ref{Dd1}. Our ultimate goal is to show a refined version of this statement.
To this end we restrict the class of DMV solutions to those that can be obtained as a limit of a consistent approximations.

\begin{Definition}[{\bf Computable DMV solution}] \label{DD1}
	
	We say that a DMV solution $\mathcal{V}_{t,x}$, $\mathfrak{C}$ is \emph{computable} if there exists a family of consistent approximations
	$(\vre, \vme, S_\ep)_{\ep > 0}$ generating $\mathcal{V}_{t,x}$, $\mathfrak{C}$
	in the sense specified in \eqref{gen1}, \eqref{gen2}.	
	
\end{Definition}

Apparently, the set of computable DMV solutions is smaller than the set of all DMV solutions. For instance, the
set of all DMV solutions starting from the same initial data is convex while convexity of the set of computable solutions is not obvious.
Next, we modify accordingly maximal dissipation criterion in the class of computable solutions.

\begin{Definition}[{\bf Maximal computable DMV solution}] \label{DDd1}	
	A DMV solution $\{ \mathcal{V}_{t,x} \}, \mathfrak{C}$ is called \emph{maximal computable}  if the associated trio of
	expected values $(\vr, \vm, S)$ is maximal with respect to the order relation
	$\prec_{\rm D}$ in the class of all computable DMV solutions emanating from the same initial data.
\end{Definition}	

Obviously, the class of maximal computable DMV solutions is larger than the class of maximal solutions, meaning there is a smaller set
of possible candidates that may dominate a given solution with respect to the order relation $\prec_{\rm D}$.
Our second main result reads:

\begin{mdframed}[style=MyFrame]
	\begin{Theorem}[{\bf Non-maximality of oscillatory approximations}] \label{cadT}~\newline
		Let $(\vre, \vme, S_\ep)_{\ep > 0}$ be an oscillatory consistent approximation generating a DMV solution of the Euler system
		with the expected values of the conservative-entropy variables $(\vr, \vm, S)$.
		
		Then $(\vr, \vm, S)$ is not a weak solution of the Euler system, not even in its generalized form \eqref{E6}--\eqref{E9}.
		Moreover,
		there exists an infinite family  $\{ \mathcal{V}^n_{t,x} \}$, $\mathfrak{C}^n$ of computable DMV solutions such that the associated expected values
		$(\vr^n, \vm^n, S^n)$ satisfy
		\[
		(\vr, \vm, S) \prec_{\rm D}
		(\vr^n, \vm^n, S^n) \prec_{\rm D} (\vr^{n+1}, \vm^{n+1}, S^{n+1}) ,\ n = 1,2,\dots
		\]	
		In particular, the limit solution is not maximal computable in the sense of Definition~\ref{DDd1}.
		
	\end{Theorem}

\end{mdframed}	

The rest of the paper is devoted to the proofs of Theorems \ref{dT1}, \ref{cadT}.

\section{Proof of Theorem \ref{dT1}}
\label{vd}

We start by rewriting the energy compatibility condition \eqref{d5} in the form
\begin{align}
	&\intO{ E(\vr_0, \vm_0, S_0) } -
	\intO{ E\Big( \big< \mathcal{V}_{t,x};\tvr \big>, \big< \mathcal{V}_{t,x};\tvm \big>, \Big< \mathcal{V}_{t,x};\widetilde{S} \Big> \Big) } \br
	\geq& \intO{ \Big[ \Big< \mathcal{V}_{t,x};E(\tvr, \tvm, \widetilde{S}) \Big> - E\Big( \big< \mathcal{V}_{t,x};\tvr \big>, \big< \mathcal{V}_{t,x};\tvm \big>, \Big< \mathcal{V}_{t,x};\widetilde{S} \Big> \Big)\Big] }	\br
	&+  r(d, \gamma) \int_{\Ov{\Omega}} \ \D {\rm trace}[ \mathfrak{C}](t, \cdot).
	\label{d6}
\end{align}	
The first term on the right-hand side,
\[
\intO{ \Big[ \Big< \mathcal{V}_{t,x};E(\tvr, \tvm, \widetilde{S}) \Big> - E\Big( \big< \mathcal{V}_{t,x};\tvr \big>, \big< \mathcal{V}_{t,x};\tvm \big>, \Big< \mathcal{V}_{t,x};\widetilde{S} \Big> \Big)\Big] }
\]
is called \emph{oscillation defect} and it is non-negative a.a. in $(0,T) \times \Omega$ by virtue of Jensen's inequality.
The second one is \emph{concentration defect} and it is non-negative because of positive semi-definiteness of the tensor $\mathfrak{C}$.
Accordingly, we call the term
\begin{align}
	\label{defect}
	D_{\rm E}(t) &= \intO{ E(\vr_0, \vm_0, S_0) } - \intO{ E\Big( \big< \mathcal{V}_{t,x};\tvr \big>, \big< \mathcal{V}_{t,x};\tvm \big>, \Big< \mathcal{V}_{t,x};\widetilde{S} \Big> \Big) } \br &\equiv
	\intO{ E(\vr_0, \vm_0, S_0) } - \intO{ E( \vr, \vm, S )(t, \cdot) }
\end{align}
\emph{energy defect}. Let us recall our agreement that the mean values
\[
\vr(t,x) = \big< \mathcal{V}_{t,x}, \tvr \big>, \ \vm(t,x) = \big< \mathcal{V}_{t,x}, \tvm \big>,\ S(t,x) = \Big< \mathcal{V}_{t,x}, \widetilde{S} \Big>
\]
are interpreted as c\` agl\` ad vector valued functions of time.

Unlike the right-hand side of \eqref{d6}, the energy defect $D_{\rm E}$ is well-defined for \emph{any} $t \geq 0$. Indeed, by virtue of weak continuity properties stated in \eqref{dd2},
convexity of the energy function $E$, and the fact $E$ is increasing in $S$, we get
\begin{align}
	&\intO{ E\Big( \big< \mathcal{V}_{\tau, x} ;\tvr \big>, \big< \mathcal{V}_{\tau, x} ;\tvm \big> , \Big< \mathcal{V}_{\tau, x} ;\widetilde{S} \Big> \Big) }\br
	\leq& \intO{ E\Big( \big< \mathcal{V}_{\tau, x} ;\tvr \big>, \big< \mathcal{V}_{\tau, x} ;\tvm \big> , \Big< \mathcal{V}_{\tau \pm, x} ;\widetilde{S} \Big> \Big) }
	\br \leq& \liminf_{t \to \tau\pm} \intO{ E\Big( \big< \mathcal{V}_{t,x};\tvr \big>, \big< \mathcal{V}_{t,x};\tvm \big>, \Big< \mathcal{V}_{t,x};\widetilde{S} \Big> \Big) }
	\notag
\end{align}
for any $\tau > 0$.
Thus we conclude that
\[
D_{\rm E}: [0, \infty) \to [0, \infty) \ \mbox{is a non-negative u.s.c. function}.
\]
Applying the same argument at $\tau = 0$, we obtain
\[
D_{\rm E}(0) = \intO{ E(\vr_0, \vm_0, S_0) } - \intO{ E( \vr_0, \vm_0, S_0 ) } = 0.
\]
Consequently, we may infer that
\begin{equation} \label{vvd1}
	D_{\rm E}: [0, \infty) \to [0, \infty) \ \mbox{is a non-negative u.s.c. function},\
	\lim_{t \to 0+} D_{\rm E}(t) = D_{\rm E}(0) = 0.
\end{equation}

\subsection{Concatenation property for DMV solutions}
\label{cp}

A crucial role in the proof of Theorem \ref{dT1} is the concatenation property
for DMV solutions.

\begin{Lemma}[{\bf Concatenation of DMV solutions}] \label{Ldv1}
	
	Let $\{ \mathcal{V}^i_{t,x} \}, \mathfrak{C}^i$ be two DMV solutions of the Euler system in $(0, T_i) \times \Omega$, with the initial data
	$(\vr_0^i, \vm_0^i, S^i_0)$, $i=1,2$ respectively.
	Suppose that
	\begin{equation} \label{cp1}
		\vr^1(T_1, \cdot) = \lim_{t \to T_1 -} \big< \mathcal{V}^1_{t,x}, \tvr \big> = \vr^2_0, \ \vm^1(T_1, \cdot) = \lim_{t \to T_1 -} \big< \mathcal{V}^1_{t, x}; \tvm \big> = \vm^2_0,
	\end{equation}	
	\begin{equation} \label{cp2}
		S^1(T_1, \cdot) = \lim_{t \to T_1-} \big< \mathcal{V}^1_{t,x}; \widetilde{S} \big>  \leq S^2_0,
	\end{equation}
	and
	\begin{equation} \label{cp3}
		\intO{ E (\vr^2_0, \vm^2_0, S^2_0 ) } \leq 	\intO{ E (\vr^1_0, \vm^1_0, S^1_0 ) }.
	\end{equation}	
	
	Then the concatenated solution $(\{ \mathcal{V}_{t,x} \}, \mathfrak{C}) = (\{\mathcal{V}^1_{t,x}\} ,\mathfrak{C}^1)  \cup_{T_1} (\{ \mathcal{V}^2_{t,x} \}, \mathfrak{C}^2) $ defined as
	\[
	\mathcal{V}_{t,x} = 
	\begin{cases}
		\mathcal{V}^1_{t,x} \ &\mbox{ for }\ 0 \leq t \leq T_1, \\
		\mathcal{V}^2_{t - T_1,x} \ &\mbox{ for }\ T_1 < t \leq T_1 + T_2
	\end{cases},\
	\mathfrak{C}(t) = 
	\begin{cases}
		\mathfrak{C}^1(t) \ &\mbox{ for }\ 0 \leq t \leq T_1, \\
		\mathfrak{C}^2(t - T_1) \ &\mbox{ for }\ T_1 < t \leq T_1 + T_2
	\end{cases}
	\]	
	is a DMV solution of the Euler system in $(0, T_1 + T_2) \times \Omega$, with the initial data $(\vr^1_0, \vm^1_0, S^1_0)$.
	
\end{Lemma}

\begin{proof}
	The proof the integral identities \eqref{d2}--\eqref{d4} is straightforward and follows from \eqref{cp1}, \eqref{cp2}. As for the energy compatibility condition \eqref{d5}, it follows directly from \eqref{cp3}.	
\end{proof}

\subsection{Vanishing energy defect for admissible solutions}
\label{proof}

We are ready to prove Theorem \ref{dT1}. Let $\{ \mathcal{V}_{t,x} \}$, $\mathfrak{C}$ be a DMV solution of the Euler system in $(0,T) \times \Omega$.
Since the total entropy
\[
t \in (0,T) \mapsto \intO{ S(t, \cdot) } = \intO{ \Big< \mathcal{V}_{t,x}, \widetilde{S} \Big> }
\]
is a non-decreasing function, it is differentiable a.a. in $(0,T)$.

Let $\tau \in (0,T)$ be a point, where $t \mapsto \intO{ \Big< \mathcal{V}_{t,x};\widetilde{S} \Big>}$ is differentiable.
This means
\begin{equation} \label{vd2}
	0 \leq	\frac{\D^+}{\dt}
	\intO{ \Big< \mathcal{V}_{t,x}; \widetilde{S} \Big> } < \infty \ \mbox{for}\ t = \tau.
\end{equation}
Suppose, in addition, that the energy defect $D_{\rm E}(\tau)$ is positive,
\[	
D_{\rm E}(\tau) > 0.
\]	
We show that such a solution cannot be maximal dissipative, which completes the proof of Theorem~\ref{dT1}. Indeed, by virtue of Lemma \ref{Ldv1}, we can concatenate this solution at the time $\tau$ with a another one emanating
from the initial data
\[
\vr^2_0 (x) = \vr(\tau, x) = \big< \mathcal{V}_{\tau, x}; \tvr \big>,\ \vm^2_0 (x) = \vm(\tau, x) = \big< \mathcal{V}_{\tau,x}; \tvm \big>
\ \mbox{for a.a.}\ x \in \Omega.
\]
and
\begin{equation} \label{vd3}
	S^2_0(x) > S(\tau, x) = \Big< \mathcal{V}_{\tau, x}; \widetilde{S} \Big>
	\ \mbox{for a.a.}\ x \in \Omega.
\end{equation}
The fact that $S^2_0$ can be chosen to satisfy the strict inequality \eqref{vd3} follows from the strict positivity of the defect. As
\begin{align}
	\intO{ E(\vr_0, \vm_0, S_0) } &> \intO{ E\Big( \big< \mathcal{V}_{\tau, x}; \tvr \big> ,  \big< \mathcal{V}_{\tau,x}; \tvm \big> , \Big< \mathcal{V}_{\tau, x} ;\widetilde{S} \Big> \Big) } \br &=  \intO{ E\Big( \vr^2_0 , \vm^2_0  , \Big< \mathcal{V}_{\tau, x} ;\widetilde{S} \Big> \Big) },\notag
\end{align}
and the total energy is increasing in $S$, we can choose $S_0$ to satisfy both \eqref{vd3}, and \eqref{cp3}, meaning,
\[
\intO{ E\big(\vr_0, \vm_0, S_0\big) } \geq \intO{ E\Big( \vr^2_0 , \vm^2_0  , S^2_0 \Big) },
\]
which is nothing other than \eqref{cp3}.
Thanks to \eqref{vd3}, the concatenated solution $\mathcal{V}_{t,x} \cup_{\tau} \mathcal{V}^2_{t,x}$ satisfies
\[
\frac{\D^+}{\dt}
\intO{ \Big< \mathcal{V}^2_{t,x}; \widetilde{S} \Big> } = \infty \ \mbox{for}\ t = \tau,
\]
meaning it dominates the original one in the sense of $\prec_{\rm D}$.
We have completed the proof of Theorem \ref{dT1}.

As the same construction can be repeated at a.a. point of positivity of
$D_{\rm E}$, we conclude there are, in fact, infinitely many ordered DMV solutions. Finally, going back to Corollary \ref{dC1}, suppose 
the Euler system admits two (weak) solutions.
Then their convex combination is a DMV solution with an energy defect $D_{\rm E}$ positive on a set of non-zero Lebesgue measure. Thus the construction described above yields the ordered sequence of DMV solutions claimed in Corollary \ref{dC1}.

\section{Proof of Theorem \ref{cadT}}
\label{ppr}

The proof of Theorem \ref{cadT} follows the same lines as that of Theorem \ref{dT1} as soon as we establish two necessary ingredients:
\begin{enumerate}
	\item Global existence of computable DMV solutions for any finite energy initial data.
	\item Showing that concatenation of two computable DMV solutions is computable.
	
\end{enumerate}

As for (1), a short inspection of the proof of \cite[Proposition 3.8]{BreFeiHof19C} reveals that the desired solution is constructed as a limit
of consistent approximations $(\vre, \vme = \vre \vue, S_\ep = \vre s_\ep)_{\ep > 0}$ solving the problem
\begin{align}
	\partial_t \vre + \Div (\vre \vue) &= 0, \br
	\partial_t (\vre \vue) + \Div (\vre \vue \otimes \vue) + \Grad p(\vre, s_\ep) &= \ep \mathcal{L}[\vue],\br
	\partial_t s_\ep + \vue \cdot \Grad s_\ep &= 0,
	\notag
\end{align}	
where $\mathcal{L}$ is a suitable 3-rd order elliptic operator.

As for (2), we show the following analogue of Lemma \ref{Ldv1}.
\begin{Lemma}[{\bf Concatenation of computable DMV solutions}] \label{LLdv1}
	
	Let $\{ \mathcal{V}^i_{t,x} \}, \mathfrak{C}^i$ be two computable DMV solutions of the Euler system in $[0, T] \times \Omega$, with the initial data
	$(\vr_0^i, \vm_0^i, S^i_0)$, $i=1,2$ respectively. Let $0 < T_1 < T$ be given. In addition,
	suppose that
	\begin{equation} \label{Lcp1}
		\vr^1(T_1, \cdot) = \lim_{t \to T_1 -} \big< \mathcal{V}^1_{t,x}, \tvr \big> = \vr^2_0, \ \vm^1(T_1, \cdot) = \lim_{t \to T_1 -} \big< \mathcal{V}^1_{t, x}; \tvm \big> = \vm^2_0,
	\end{equation}	
	\begin{equation} \label{Lcp2}
		S^1(T_1, \cdot) = \lim_{t \to T_1-} \Big< \mathcal{V}^1_{t,x}; \widetilde{S} \Big>  \leq S^2_0,
	\end{equation}
	and
	\begin{equation} \label{Lcp3}
		\intO{ E (\vr^2_0, \vm^2_0, S^2_0 ) } \leq 	\intO{ E (\vr^1_0, \vm^1_0, S^1_0 ) }.
	\end{equation}	
	
	Then there exists at most countable set $\mathcal{T} \subset (0,T)$ such that for any
	$T_1 \in (0,T) \setminus \mathcal{T}$,
	the concatenated solution  $(\{ \mathcal{V}_{t,x} \}, \mathfrak{C}) = (\{\mathcal{V}^1_{t,x}\} ,\mathfrak{C}^1)  \cup_{T_1} (\{ \mathcal{V}^2_{t,x} \}, \mathfrak{C}^2) $
	defined as
	\[
	\mathcal{V}_{t,x} = 
	\begin{cases}
		\mathcal{V}^1_{t,x} \ &\mbox{ for }\ 0 \leq t \leq T_1, \\
		\mathcal{V}^2_{t - T_1,x} \ &\mbox{ for }\ T_1 < t \leq T_1 + T
	\end{cases},\
	\mathfrak{C}(t) = 
	\begin{cases}
		\mathfrak{C}^1(t) \ &\mbox{ for }\ 0 \leq t \leq T_1, \\
		\mathfrak{C}^2(t - T_1) \ &\mbox{ for }\ T_1 < t \leq T_1 + T
	\end{cases}
	\]	
	is again a computable DMV solution of the Euler system in $[0, T_1 + T] \times \Omega$, with the initial data $(\vr^1_0, \vm^1_0, S^1_0)$.
	
\end{Lemma}

\begin{proof}
	
	The idea is to construct the generating sequence by concatenating the generating sequences $(\vr^i_\ep, \vm^i_\ep, S^i_\ep)_{\ep > 0}$, $i=1,2$.
	Accordingly, we define
	\[
	(\vre, \vme, S_\ep) (t,\cdot) = 
	\begin{cases}
		(\vre^1, \vme^1, {S}_\ep^1)(t-, \cdot) \ &\mbox{ for }\ 0 \leq t \leq T_1 \\
		(\vre^2, \vme^2, S^2_\ep ((t - T_1)-, \cdot) \ &\mbox{ for }\ T_1 < t \leq T_1 + T.
	\end{cases}
	\]
	
	\textcolor{black}{Let us denote 
		\[
		e^{i,j}_{\ep}, \ \ i=1,2,3,4 \mbox{ and } j=1,2
		\]
		the consistency errors $e_\ep^i$, $i=1,2,3,4$ defined in \eqref{ca2}--\eqref{ca5} corresponding to the solution $\{ \mathcal{V}^j_{x,t} \}, $ $\mathfrak{C}^j$, $j=1,2$.  Similarly, the consistency error associated to the concatenated solution is denoted by $\tilde{e}_\ep^i$.
	}
	
	\noindent {\bf Step 1:}
	
	First observe that the approximate energy inequality \eqref{ca2} is satisfied with the consistency error
	\textcolor{black}{
		\[
		\tilde{e}^1_\ep = \max \{e^{1,1}_\ep, e^{1,2}_\ep \}.
		\]
	}
	Indeed
	\begin{align}
		\intO{ E(\vre, \vme, S_\ep )(\tau, \cdot) } &= \intO{ E(\vre^1, \vme^1, S^1_\ep )(\tau, \cdot) } \br
		&\leq \intO{ E(\vr_0^1, \vm_0^1, S_0^1) } + {\color{black} e^{1,1}_\ep} \ \mbox{ for a.a.}\ \tau \in (0,T_1),\notag
	\end{align}
	while
	\begin{align}
		\intO{ E(\vre, \vme, S_\ep )(\tau, \cdot) } &= \intO{ E(\vre^2, \vme^2, S^2_\ep )(\tau - T_1, \cdot) }\br
		&\leq \intO{ E(\vr^2_0, \vm^2_0, S^2_0) } + {\color{black} e^{1,2}_\ep} \br
		&\leq \intO{ E(\vr^1_0, \vm^1_0, S^1_0) } + {\color{black} e^{1,2}_\ep} 
		\quad \mbox{ for a.a.}\ \tau \in (T_1,T_1 + T).
		\notag
	\end{align}
	
	\noindent {\bf Step 2:}
	
	Next, it is easy to check that the concatenated density $\vre$ satisfies the approximate equation of continuity \eqref{ca3} with the error
	\textcolor{black}{
		\begin{align}
			\tilde{e}^2_\ep [\varphi] &= \mathds{1}_{[0, T_1)} e^{2,1}_\ep [\varphi] + \mathds{1}_{[T_1, T_1 + T]} e^{2,2}_\ep (\cdot - T_1)[\varphi]  - \delta_{T_1}
			\intO{ \Big( \vr^1_\ep (T_1 -, \cdot) - \vr^2_0 \Big) \varphi } \br &=
			\notag \mathds{1}_{[0, T_1)} e^{2,1}_\ep [\varphi] + \mathds{1}_{[T_1, T_1 + T]} e^{2,2}_\ep (\cdot - T_1)[\varphi]  - \delta_{T_1}
			\intO{ \Big( \vr^1_\ep (T_1 -, \cdot) - \vr^1(T_1, \cdot) \Big) \varphi }.
		\end{align}
	}
	By virtue of \eqref{ca3},
	\[
	\intO{ \Big( \vr^1_\ep (T_1 \pm, \cdot) - \vr^1(T_1, \cdot) \Big) \varphi } \to 0 \ \mbox{as}\ \ep \to 0,
	\]
	and
	we may infer $(\vre, \vme)$ satisfy the approximate equation of continuity \eqref{ca3}. The same argument applies
	to the momentum equation \eqref{ca4}.
	
	\noindent {\bf Step 3:}
	As the concatenated consistent approximations obviously satisfy the entropy minimum principle \eqref{mini}, it remains to
	verify the entropy inequality \eqref{ca5}. Arguing as in Step~2 we deduce that concatenation produces an extra term
	in the entropy jump, namely
	\[
	\delta_{T_1} \intO{ \Big( S^2_0 - S^1_\ep (T_1-, \cdot) \Big) \varphi }, \ \varphi \geq 0
	\]
	on the right--hand side of the approximate entropy inequality \eqref{ca5}. Moreover, in view of \eqref{Lcp2},
	\[
	\delta_{T_1} \intO{ \Big( S^2_0 - S^1_\ep (T_1-, \cdot) \Big) \varphi } \geq \delta_{T_1} \intO{ \Big( S^1(T_1, \cdot) - S^1_\ep (T_1-, \cdot) \Big) \varphi }.
	\]
	Finally, as $(\vr^1_\ep, \vm^1_\ep, S^1_\ep)$ satisfies \eqref{ca5} in $(0,T)$, Helly's theorem implies
	\[
	\intO{ \Big( S^1(T_1, \cdot) - S^1_\ep (T_1-, \cdot) \Big) \varphi } \to 0
	\]	
	with a possible exception of a countable set $T_1 \in \mathcal{T}$ --
	the set of discontinuities of $t \mapsto \intO{ S^1 (t, \cdot) \varphi }$.
\end{proof}

With Lemma \eqref{LLdv1} at hand, the proof of Theorem \ref{cadT} can be performed in the same way as in Section \ref{proof}. Indeed the set of
times $\tau$ for which both \eqref{vd2} and concatenation of computable solutions hold is still of full measure in $(0,T)$.

\section{DMV solutions with small energy defect}
\label{def}

\textcolor{black}{Using the concatenation principle introduced in Section~\ref{cp}} we show  that there exist DMV solutions with arbitrarily small energy defect.

\begin{mdframed}[style=MyFrame]
	
	\begin{Theorem}[{\bf DMV solutions with small defect}] \label{Tdef}
		Let $\Omega \subset R^d$  and the initial data $(\vr_0, \vm_0, S_0)$ satisfy the hypotheses of Proposition \ref{Pd1}.
		Let $\delta > 0$ be given.
		
		Then there exists a DMV solution satisfying
		\[
		0 \leq D_{\rm E}(t) = \intO{ E(\vr_0, \vm_0, S_0) } - \intO{ E(\vr, \vm, S)(t, \cdot) } \leq \delta \ \mbox{for a.a.}\ t \in (0,T).
		\]	
		
	\end{Theorem}	
	
\end{mdframed}

\begin{proof}
	The proof is similar to \cite[Section 3]{FeLmJu2025} based on Zorn's lemma (Axiom of Choice).
	Denote
	\begin{align}
		\mathcal{U}[\vr_0, \vm_0, S_0] &= \Big\{ (\vr, \vm, S) \ \mid
		\vr= \big< \mathcal{V}_{t,x}, \tvr \big>, \ \vm = \big< \mathcal{V}_{t,x}, \tvm \big>,\ S = \Big< \mathcal{V}_{t,x}, \widetilde{S} \Big>  \br
		&\mbox{expected values of a DMV solution starting from}\ (\vr_0, \vm_0, S_0)
		\Big\}.
		\notag
	\end{align}
	Let
	\[
	T(\delta)[\vr, \vm, S] = \sup \big\{ \tau \in [0,T]\ \mid \norm*{D_{\rm E}}_{L^\infty(0,\tau)} \leq \delta \big\}.
	\]
	We know from \eqref{vvd1} that $T(\delta) [\vr, \vm, S] > 0$.
	
	Next, we introduce the order relation
	\begin{align} \label{ordering}
		(\vr^1, \vm^1, S^1) &\prec \hspace{-0.2cm} \prec (\vr^2, \vm^2, S^2) \br
		&\Leftrightarrow_{\rm{def}} \br
		(\vr^2, \vm^2, S^2) &=  (\vr^1, \vm^1, S^1) \ \mbox{for}\ t \in [0, T(\delta)[\vr^1, \vm^1, S^1]].
	\end{align}
	It is easy to check that $\mathcal{U}[\vr_0, \vm_0, S_0]$ with the relation $\prec \hspace{-0.2cm} \prec$ is a partially ordered set. \textcolor{black}{Strictly speaking, the relation $\prec \hspace{-0.2cm} \prec$ is partial ordering on  the equivalent classes
		\begin{align*}
			(\vr^1, \vm^1, S^1 ) &\sim (\vr^2, \vm^2, S^2 ) \\
			&\Leftrightarrow_{\rm{def}} \\
			T(\delta)[\vr^1, \vm^1, S^1] &= T(\delta)[\vr^2, \vm^2, S^2] \br \mbox{ and } \quad
			(\vr^1, \vm^1, S^1 )(t,\cdot) &= (\vr^2, \vm^2, S^2 )(t, \cdot)  \mbox{ for } 0 \leq t \leq T(\delta).
		\end{align*}
	}
	
	Let
	\[
	(\vr^1, \vm^1, S^1) \prec \hspace{-0.2cm} \prec \dots  \prec \hspace{-0.2cm} \prec (\vr^n, \vm^n, S^n) \prec \hspace{-0.2cm} \prec (\vr^{n+1}, \vm^{n+1}, S^{n + 1})
	\]
	be an ordered sequence. Thanks to compactness of the set $\mathcal{U}[\vr_0, \vm_0, S_0]$, there is a subsequence (not relabeled) such that
	\begin{align}
		\vr^n \to \vr,\ S^n \to S \ &\mbox{weakly-(*) in}\ L^\infty(0,T; L^\gamma(\Omega)),\br
		\vm^n \to \vm \ &\mbox{weakly-(*) in}\ L^\infty(0,T; L^{\frac{2 \gamma}{\gamma + 1}}(\Omega; R^d)),\notag
	\end{align}
	where the limit $(\vr, \vm, S)$ belongs to	$\mathcal{U}[\vr_0, \vm_0, S_0]$. As
	\[
	(\vr^n, \vm^n, S^n)(t, \cdot) = (\vr^k, \vm^k, S^k) (t,\cdot) \ \mbox{in}\ [0, T(\delta)[\vr^k, \vm^k, S^k]]
	\ \mbox{for any fixed}\ k \ \mbox{and all}\ n \geq k,
	\]
	\textcolor{black}{we conclude that
		\[
		(\vr, \vm, S)(t, \cdot) = (\vr^k, \vm^k, S^k) (t,\cdot) \ \mbox{in}\ [0, T(\delta)[\vr^k, \vm^k, S^k]].
		\]
		Consequently,
		\[
		(\vr^k, \vm^k, S^k) \prec \hspace{-0.2cm} \prec(\vr, \vm, S) \ \mbox{for any}\ k =1,2,\dots.
		\]}
	By virtue of Zorn's lemma, the set $\mathcal{U}[\vr_0, \vm_0, S_0]$ contains an element $(\tvr, \tvm, \widetilde{S})$ maximal with respect
	to the order relation $\prec \hspace{-0.2cm} \prec$. 
	We claim
	\[
	T(\delta)[\tvr, \tvm, \widetilde{S}] = T
	\]
	which completes the proof. Assuming the contrary, meaning
	\[
	T(\delta)[\tvr, \tvm, \widetilde{S}] = T_1 < T,
	\]
	we concatenate the solution $(\tvr, \tvm, \widetilde{S})$ with $(\vr^2, \vm^2, S^2)$ at $T = T_1$ satisfying
	\[
	\vr^2_0 = \tvr(T_1, \cdot), \ \vm^2_0 = \tvm(T_1, \cdot),\ S^2_0 \geq \widetilde{S}(T_1, \cdot),
	\]
	where $S^2_0$ is chosen in such a way that
	\textcolor{black}{
		\[
		\intO{ E(\vr_0, \vm_0, S_0) } = \intO{ E(\vr^2_0, \vm^2_0, S^2_0) } \geq \intO{ E(\tvr, \tvm, \widetilde{S})(T_1, \cdot) },
		\]
	}
	cf. Lemma \ref{Ldv1}. Obviously, the concatenated solution
	\[
	(\vr, \vm, S)(t, \cdot) = 
	\begin{cases}
		(\tvr, \tvm, \widetilde{S}) (t,\cdot) \ &\mbox{ for }\ 0 \leq t \leq T_1, \\
		(\vr^2, \vm^2, S^2)(t - T_1) \ &\mbox{ for }\ t > T_1
	\end{cases}
	\]
	satisfies $(\tvr, \tvm, \widetilde{S}) \prec \hspace{-0.2cm} \prec (\vr, \vm, S)$, and, by virtue of \eqref{vvd1},
	\[
	T(\delta) [\vr, \vm, S] > T_1 = T(\delta) [\tvr, \tvm, \widetilde{S}]
	\]
	in contrast with maximality of $(\tvr, \tvm, \widetilde{S})$.
	
\end{proof}	

\section{Energy preserving measure-valued solutions}
\label{B}	

Motivated by Fjordholm, Mishra, and Tadmor \cite{FjMiTa1}, we introduce a class of energy preserving DMV solutions based on a single parametrized (Young) measure $\{ \mathcal{V}_{t,x} \}$ sitting on the space of dummy variables
\[
\mathcal{V}_{t,x} \in \mathfrak{P} \Big\{ (\tvr, \tvm, \tvE) \mid
\ (\tvr, \tvm, \tvE) \in R^{d+2} \Big\}
\]
for a.a. $(t,x) \in (0,T) \times \Omega$. We tacitly assume that the
measures $\mathcal{V}_{t,x}$ are sufficiently localized in the phase space
to accommodate all non-linear fluxes appearing in the Euler system. Thus we allow for a larger framework than \cite{FjMiTa1}, where the Young measure is generated by bounded sequences of approximate solutions.
In view of the recent results of Buckmaster et al.
\cite{BuCLGS}, Cao-Labora et al. \cite{CLGSShiSta},
Merle et al. \cite{MeRaRoSz},
\cite{MeRaRoSzbis}, uniformly bounded consistent approximations may not exist at least for certain initial data. This follows from the following
arguments:
\begin{itemize}
	\item By the aforementioned results, the \emph{isentropic} Euler system admits smooth solutions that blow up in the $L^\infty$-norm in a finite time.
	\item Smooth solutions of the isentropic Euler system are smooth solutions
	of the full Euler system \eqref{E1}--\eqref{E3} with constant entropy.
	\item By the weak--strong uniqueness principle, any DMV solution of the
	Euler system coincides with the smooth solution as long as the latter exists.
	
\end{itemize}	

As a consequence of the above arguments, the existence of global-in-time DMV solution with uniformly bounded support of Young measure is \emph{incompatible} with
the existence of smooth solutions exhibiting a finite-time blow-up.

\begin{Definition} [{\bf Energy preserving DMV solution}] \label{DB1}
	\emph{Energy preserving DMV solution} of the Euler system
	\eqref{E1}--\eqref{E3}, \eqref{E4}, with the initial data $(\vr_0, \vm_0, E_0)$ in $(0,T) \times \Omega$
	is a parametrized measure
	\[
	\mathcal{V}_{t,x} \in \mathfrak{P} \Big\{
	(\tvr, \tvm, \tvE) \mid \ (\tvr, \tvm, \tvE) \in R^{d+2} \Big\},\
	\mathcal{V} \in L^\infty_{\rm weak-(*)} ((0,T) \times \Omega; \mathfrak{P}(R^{d+2}) )
	\]	
	satisfying:
	
	\begin{itemize}	
		
		\item
		The equation of continuity
		\begin{equation} \label{Bd2}
			\int_0^T \intO{ \Big[ \big< \mathcal{V}_{t,x}; \tvr \big> \partial_t \varphi(t,x) + \big< \mathcal{V}_{t,x}; \tvm \big> \cdot \Grad \varphi(t,x) \Big]} \dt = - \intO{ \vr_0(x) \varphi(0, x) } 	
		\end{equation}
		holds for any $\varphi \in C^1([0,T) \times \Ov{\Omega})$.
		\item
		The momentum equation
		\begin{align}
			\int_0^T &\intO{ \bigg[ \big< \mathcal{V}_{t,x}; \tvm \big> \cdot \partial_t \bfphi(t,x)  + \bigg< \mathcal{V}_{t,x};\mathds{1}_{\tvr > 0} \bigg(
				\frac{ \tvm \otimes \tvm}{\tvr} + p(\tvr, \widetilde{E}) \mathbb{I} \bigg)  \bigg> : \Grad \bfphi(t,x) 	 \bigg] } \dt \br
			&=
			\label{Bd3}	- \intO{ \vm_0(x) \cdot \bfphi(0, x) }
		\end{align}
		holds for any $\bfphi \in C^1([0,T) \times \Ov{\Omega}; R^d)$, $\bfphi \cdot \bm{n}|_{\partial \Omega} = 0$.
		\item
		The energy equality
		\begin{align}
			\int_0^T &\intO{ \bigg[ \Big< \mathcal{V}_{t,x}; \tvE\Big>  \partial_t \varphi(t,x)  + \bigg< \mathcal{V}_{t,x}; \mathds{1}_{\tvr > 0} \Big( \tvE + p(\tvr, \tvE) \Big) \frac{\tvm}{\tvr} \bigg> \cdot \Grad \varphi \bigg] } \dt \br &= - \intO{ E_0 \varphi (0, \cdot) }		
			\label{Bd3a}
		\end{align}
		holds for any $\varphi \in C^1([0,T) \times \Ov{\Omega})$.
		\item
		The entropy inequality
		\begin{align}
			\int_0^T &\intO{ \bigg[ \Big< \mathcal{V}_{t,x}; S(\tvr, \tvm, \tvE) \Big> \partial_t \varphi(t,x) + \bigg< \mathcal{V}_{t,x};\mathds{1}_{\tvr > 0} S(\tvr, \tvm, \tvE) \frac{\tvm}{\tvr} \bigg>\cdot \Grad \varphi(t,x) \bigg] } \dt
			\br &\leq - \intO{ S(\vr_0, \vm_0, E_0) \varphi (0, x) }
			\label{Bd4}		
		\end{align}
		holds for any $\varphi \in C^1([0,T) \times \Ov{\Omega})$, $\varphi \geq 0$.
		\item The compatibility condition
		\begin{align}
			\mathcal{V}_{t,x} \bigg\{ \tvr \geq 0  \lor \tvE - \frac{1}{2}\frac{|\tvm|^2}{\tvr} \geq 0 \lor S(\tvr, \tvm, \tvE) \geq \Ov{s} \tvr \bigg\} = 1	
			\label{Bd5}
		\end{align}	
		holds for a.a. $(t,x) \in (0,T) \times \Omega$.
	\end{itemize}

\end{Definition}

In the above definition, we tacitly assume that all integrals are well-defined, meaning the integrand are at least integrable in $(0,T) \times \Omega$. The kinetic energy $\frac{1}{2} \frac{|\vm|^2}{\vr}$ is understood as a convex l.s.c function $(\vr, \vm) \in R^{d+1}$ specified
by \eqref{RR1}. The entropy defined as
\[
S(\vr, \vm, E) =
\begin{cases}
	\vr \log \Bigg( (\gamma - 1)
	\frac{ E - \frac{1}{2} \frac{|\vm|^2 }{\vr} }{\vr^\gamma}  \Bigg) \ &\mbox{if}\ \vr > 0,\ E > \frac{1}{2} \frac{|\vm|^2 }{\vr},\\
	0 \ &\mbox{if}\ \vr = 0, \vm = 0,\ E \geq 0 , \\
	- \infty \ &\mbox{otherwise}
\end{cases}
\]
is a concave u.s.c function of the variables $(\vr, \vm, E)$.

Similarly to the preceding part, we define the total entropy as
\[
\intO{ \widetilde{S}(t, \cdot) } = \lim_{\delta \to 0} \int_{t- \delta}^t \intO{ \Big< \mathcal{V}_{t,x};S(\tvr, \tvm, \tvE)  \Big> } ,\ \intO{ \widetilde{S}(0, \cdot) } = \intO{ S(\vr_0, \vm_0, E_0) }.
\]
In view of \eqref{Bd4}, we observe that the total entropy
\[
t \mapsto \intO{ \widetilde{S}(t, \cdot) }
\]
is a non-decreasing c\` agl\` ad function of $t \in [0,T]$. Moreover,
as $S = S(\vr, \vm, E)$ is a concave function of its arguments, we get
\[
\limsup_{t \to 0+} \intO{ \widetilde{S}(t, \cdot) } \leq
{\rm ess} \limsup_{t \to 0+} \intO{ S(\vr, \vm, E)(t, \cdot) } \leq
\intO{ S(\vr_0, \vm_0, E_0)};
\]
whence
\begin{equation} \label{Bd6}
	\lim_{t \to 0+} \intO{ \widetilde{S}(t, \cdot) } =
	\intO{ S(\vr_0, \vm_0, E_0) }.
\end{equation}

Introducing the expected values
\[
\vr(t,x) =\big< \mathcal{V}_{t,x}, \tvr \big>, \ \vm(t,x) = \big< \mathcal{V}_{t,x}, \tvm \big>,\ E(t,x) = \Big< \mathcal{V}_{t,x}, \tvE \Big>, 
\]
we may formulate Dafermos' admissibility criterion in terms of the order relation
\begin{align}
	(\vr_1, \vm_1, E_1) &\prec_{\rm D} (\vr_2, \vm_2,  E_2) \br &\Leftrightarrow
	\br \mbox{there exists} \ \tau \geq 0 \ \mbox{such that}\
	(\vr_1, \vm_1, E_1) (t, \cdot) &= (\vr_2, \vm_2,  E_2) (t, \cdot)
	\ \mbox{for all}\ 0 \leq t \leq \tau, \br
	\frac{\D^+}{\dt} \intO{ \widetilde{S}_2 (\tau, \cdot) } &>
	\frac{\D^+}{\dt}  \intO{ \widetilde{S}_1 (\tau, \cdot) }.
	\notag
\end{align}

\begin{Definition}[{\bf Maximal energy preserving DMV solution}] \label{DB2}
	We say that an energy preserving DMV solution is maximal if it is maximal with respect to the relation $\prec_{\rm D}$ in the class of all energy preserving DMV solutions.

\end{Definition}

\begin{mdframed}[style=MyFrame]
	
	\begin{Theorem}[{\bf Regularity of energy preserving maximal DMV solutions}] \label{TB1}~\newline
		Suppose that the Euler system admits an energy preserving DMV solution for any initial data $(\vr_0, \vm_0, E_0) \in L^1(\Omega; R^{d+2})$ satisfying
		\[
		\vr_0 \geq 0,\ E_0 - \frac{1}{2} \frac{|\vm_0|^2 }{\vr_0} \geq 0,\
		S(\vr_0, \vm_0, E_0) \geq \Ov{s} \vr_0.
		\]
		
		Let $\{ \mathcal{V}_{t,x} \}$ be a maximal energy preserving DMV solution of the Euler system in $(0,T) \times \Omega$.
		
		Then
		\[
		\mathcal{V}_{t,x} = \delta_{[\vr(t,x), \vm(t,x), E(t,x)]} \ \ \mbox{for a.a.}\ (t,x) \in (0,T) \times \Omega,
		\]
		where $(\vr, \vm, E)$ is a weak solution of the Euler system.
	\end{Theorem}
	
\end{mdframed}

\begin{proof}
	
	The proof is similar to that of Theorem \ref{dT1}. The oscillations
	of the solution are characterized by the entropy defect
	\begin{align}
		D_{\rm Ent}(t) &= \intO{ S(\vr, \vm, E)(t, \cdot) } -
		\intO{ \Big< \mathcal{V}_{t,x};S(\tvr, \tvm, \tvE) \Big> } \br
		&= \intO{ S(\vr, \vm, E)(t, \cdot) } -
		\intO{ \widetilde{S}(t, \cdot) } \geq 0,
		\notag
	\end{align}
	which is a well-defined finite quantity for a.a. $t \in (0,T)$.
	
	If the parametrized measure $\{ \mathcal{V}_{t,x} \}$ is not a Dirac mass a.a. in $(0,T) \times \Omega$, we can find
	\begin{align}
		0 < T_1 < T,\  \intO{ S(\vr, \vm, E)(T_1, \cdot) } &-
		\intO{ \widetilde{S}(T_1, \cdot) } > 0, \br
		\frac{\D^+ }{\dt} \intO{ \widetilde{S} (T_1, \cdot) } &=
		\frac{\D }{\dt} \intO{ \widetilde{S} (T_1, \cdot) } < \infty.
		\notag
	\end{align}
	Now, we define a new solution by concatenation $\mathcal{V}_{t,x} \cup_{T_1}
	\mathcal{V}^1_{t,x}$
	of $\{ \mathcal{V}_{t,x} \}$ at the time $T_1$ with a solution $\{ \mathcal{V}^1_{t,x} \}$ starting form the initial data
	\[
	\vr^1_0 = \vr(T_1, \cdot),\ \vm^1_0 = \vm (T_1, \cdot),\
	E^1_0 = E(T_1, \cdot),
	\] 	
	where
	\[
	\intO{ S(\vr^1_0, \vm^1_0, E^1_0) } > \intO{\widetilde{S}(T_1)}.
	\]
	Denoting the concatenated solution $\{ \mathcal{V}^2_{t,x} \} =
	\{ \mathcal{V}_{t,x} \} \cup_{T_1} \{ \mathcal{V}^1_{t,x} \}$, we have
	\[
	\frac{\D^+}{\dt} \intO{ \widetilde{S}^2(T_1, \cdot)} = \infty;
	\]
	whence the original solution is not $\prec_{\rm D}$ maximal.
	
\end{proof}

We conclude this part by showing an analogy of the convergence dichotomy stated in Proposition \ref{Pca1}. To this end, we first introduce a new order relation proposed by DiPerna \cite{DiP2} in the class of energy preserving DMV solutions: 
\begin{align}
	\{ \mathcal{V}^1_{t,x} \} &\prec_{\rm DP} \{ \mathcal{V}^2_{t,x} \} \br
	&\Leftrightarrow_{\rm def}   \br 
	\intO{ \Big< \mathcal{V}^1_{t,x}; S(\tvr, \tvm, \tvE) \Big> } &\leq  	\intO{ \Big< \mathcal{V}^2_{t,x}; S(\tvr, \tvm, \tvE) \Big> } \ \mbox{for a.a.}\ t \in (0,T).
	\notag
\end{align}	
We say that a solution $\{ \widetilde{\mathcal{V}}_{t,x} \}$ is 
$\prec_{\rm DP}$ maximal if for any other solution $\{ \mathcal{V}_{t,x} \}$ satisfying 
\[
\{ \widetilde{\mathcal{V}}_{t,x} \} \prec_{\rm DP} \{ \mathcal{V}_{t,x} \}
\]
there holds	
\[
\intO{ \Big< \mathcal{V}_{t,x}; S(\tvr, \tvm, \tvE) \Big> } =  	\intO{ \Big< \widetilde{\mathcal{V}}_{t,x}; S(\tvr, \tvm, \tvE) \Big> } \ \mbox{for a.a.}\ t \in (0,T).
\]

Let
\[
\vr =\big< \mathcal{V}_{t,x}, \tvr \big>, \ \vm= \big< \mathcal{V}_{t,x}, \tvm \big>,\ E= \Big< \mathcal{V}_{t,x}, \tvE \Big>
\]	
be the expected values of a $\prec_{\rm DP}$ maximal energy preserving 
DMV solution of the Euler system. By Jensen's inequality, we have 
\[
\intO{ \Big< \mathcal{V}_{t,x};  S(\tvr, \tvm, \tvE)(t,\cdot) \Big> } \leq 
\intO{ S(\vr, \vm, E) (t,\cdot) }.
\]
As $S$ is (strictly) concave, we immediately observe: 

The expected values $(\vr, \vm, E)$ represent a weak solution of the Euler system if and only if 
\[
\mathcal{V}_{t,x} = \delta_{[\vr(t,x), \vm(t,x), E(t,x)] } 
\ \mbox{for a.a.}\ (t,x) \in (0,T) \times \Omega.
\]

\section{Numerical simulations}
\label{NUM}
To illustrate our theoretical results, we present numerical simulations obtained by a structure-preserving numerical method. It yields stable and consistent approximations that converge, up to a subsequence as the case may be, to a DMV solution $(\vr, \vm, S)$ of the Euler system, cf.~\eqref{gen1}.

To start, a computational domain $\Omega \subset R^d$ will be discretized by a mesh $\mathcal{T}_h$, where $h \in (0,1)$ denotes a mesh parameter. Here, we consider a regular rectangular grid $\Omega = \bigcup_{K \in \mathcal{T}_h }  K $, with mesh cells $ |K | = \mathcal{O}(h^d)$. Cell interfaces are denoted by $\sigma $, $|\sigma| = \mathcal{O}(h^{d-1})$.  Let  $r_h$ be a piecewise constant function on $\mathcal{T}_h$.  Then,   $\jump{ r_h }$ and $\avs{r_h}$ stand for standard jump and average discrete operators on a cell interface $\sigma $, respectively.

Specifically, we will work with the viscous finite volume method (VFV), see, e.g., \cite{FeLMMiSh}, for further details on its description and analysis,
\begin{equation}
	\begin{aligned} 
		\label{FV}
		D_t \vr_h|_K &=-\frac{1}{|K|}\sum_{\sigma \in \partial K} \int_\sigma F_{h, \vr}^{\mathrm{up}}[\vr_h, \vu_h] \mathrm{d} S_x, \\ 
		D_t \vm_h|_K &=-\frac{1}{|K|}\sum_{\sigma \in \partial K} \int_\sigma \Big(  \boldsymbol{F}_{h,\vm}^{\mathrm{up}}(\vm_h, \vu_h) + \avs{p_h} \vn -h^{\alpha-1}  \jump{ \vu_h } \Big) \mathrm{d} S_x,\\
		D_t E_h|_K &= -\frac{1}{|K|}\sum_{\sigma \in \partial K} \int_\sigma \Big( F_{h, E}^{\mathrm{up}}[E_h, \vu_h] 
		+\avs{p_h}\vu_h\cdot \vn
		+ p_h  \avs{\vu_h}  \cdot \vn \\
		& -h^{\alpha-1}   \jump{ \vu_h } \cdot \avs{\vu_h} \Big)  \mathrm{d} {S}_x.
	\end{aligned}
\end{equation}
On any cell interface $\sigma \in \partial K$ we apply the \emph{viscous upwind numerical flux} 
\begin{equation}
	\label{flux}
	\begin{aligned}
		&\Big( F_{h, \vr}^{\mathrm{up}}, \boldsymbol{F}_{h,\boldsymbol{m}}^{\mathrm{up}}, F_{h, E}^{\mathrm{up}}\Big) =U p
		\big[\vU_h, \vu_h\big]-h^{\varepsilon}\jump{ \vU_h } ,  \\
		&U p[\vU_h, \vu_h]  =\avs{\vU_h}\avs{\vu_h} \cdot \vn-\frac{1}{2} \big|\avs{\vu_h} \cdot \vn \big| \jump{ \vU_h }, 
	\end{aligned}
\end{equation} 
where $\vU_h = (\vr_h, \vm_h, E_h)$ and $\vn$  denotes the outer normal to $\sigma \in \partial K$.  In numerical simulations presented below we set numerical diffusion parameters $\alpha=1.0$ and  $\varepsilon=2.0$.\footnote{For the fifth and seventh-order VFV methods used below we set $\alpha=5.0$ and $\varepsilon=6.0$.}

Numerical solutions $\{ (\vr_h, \vm_h, E_h) \}_{h \searrow 0}$ are obtained from \eqref{FV}, \eqref{flux} by piecewise constant interpolation in space
\footnote{Time approximation is realized by a suitable ODE solver. For consistency we apply a third-order strong stability preserving Runge Kutta method for  both the first-order VFV method as well as for higher-order VFV methods.}. 

As proved in \cite[Chapter~10]{FeLMMiSh} \eqref{FV}, \eqref{flux} yield a consistent approximation, in the sense of Definition~\ref{Dca1}, which is convergent to a DMV solution. 
More precisely, in addition to the weak convergence \eqref{gen1}, the strong convergence of the so-called Ces\`aro averages holds \cite[Theorem~10.5]{FeLMMiSh}
\begin{equation*}
	\begin{split}
		\frac{1}{N} \sum_{n=1}^N (\vr_{h_n}, \vm_{h_n}, S_{h_n} ) &\to (\vr, \vm, S) \ \ \mbox{ as }\ N \to \infty \ \ \mbox{ in } \ L^q((0,T) \times \Omega;  R^{d+2})
	\end{split}
\end{equation*}
for any $ 1 \leq q < \infty$. Further,  
\begin{equation*}
	d_{W_s} \Bigg[ \frac{1}{N} \sum_{n=1}^N \delta_{[\vr_{h_n},
		\vm_{h_n}, S_{h_n} ]} ; \mathcal{V} \Bigg] \to 0
	\ \mbox{as}\ N \to \infty \ \ \mbox{in}\ L^q((0,T) \times \Omega),
\end{equation*}
where $d_{W_s}$ denotes the Wasserstein distance, $1 \leq s, q < 2$.   
Thus, the corresponding Young measure $\mathcal{V}_{t,x}$ is a computable DMV solution to the Euler system \eqref{E1}--\eqref{E3}. 

In what follows, we demonstrate that numerical solutions obtained by the VFV method are oscillatory in the sense of Definition~\ref{Dac2}. 
Specifically, we consider the following initial data for the so-called Kelvin--Helmholtz problem on $\Omega = [0,1]^2$:
\begin{equation}\label{euler:KH_ini}
	(\rho,u,v,p)(x,y,0) = 
	\begin{cases}
		(2,-0.5,0,2.5),\ &\text{if} \  I_1\leq y \leq I_2, \\
		(1,\ \ 0.5,0,2.5),\ &\text{otherwise},\\
	\end{cases}
\end{equation}
where the interface profiles $I_j = I_j(x) = J_j + \epsilon Y_j(x), j = 1,2$ are chosen to be small perturbations around the lower interface with $ y= J_1 = 0.25$ and the upper interface with $y = J_2 = 0.75$, respectively. Further,  
\[Y_j(x) = \sum_{k=1}^m a_j^k \cos(b_j^k+2k\pi x), \ j = 1,2, \]
where $a_j^k\in[0,1]$ and $b_j^k, \ j=1,2,k=1,\cdots,m$ are arbitrary but fixed numbers. The coefficients $a^k_j$ have been normalized such that $\sum_{k=1}^m a_j^k=1$ to guarantee that $|I_j(x) - J_j|\leq \epsilon$ for $j = 1,2$. Here,  the parameters are set  $m=10$, $\epsilon=0.01$. The ratio of specific heats is $\gamma=1.4$.

\begin{figure}[!h]
	\centering
	\subfigure[first-order VFV method]{
		\includegraphics[width=0.30\linewidth]{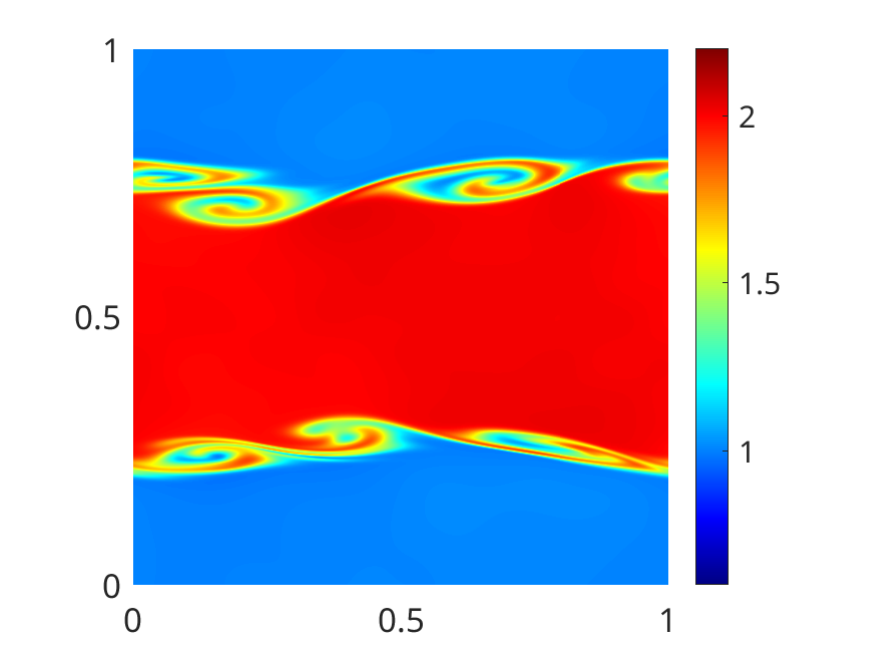}
	}
	\subfigure[second-order VFV method]{
		\includegraphics[width=0.30\linewidth]{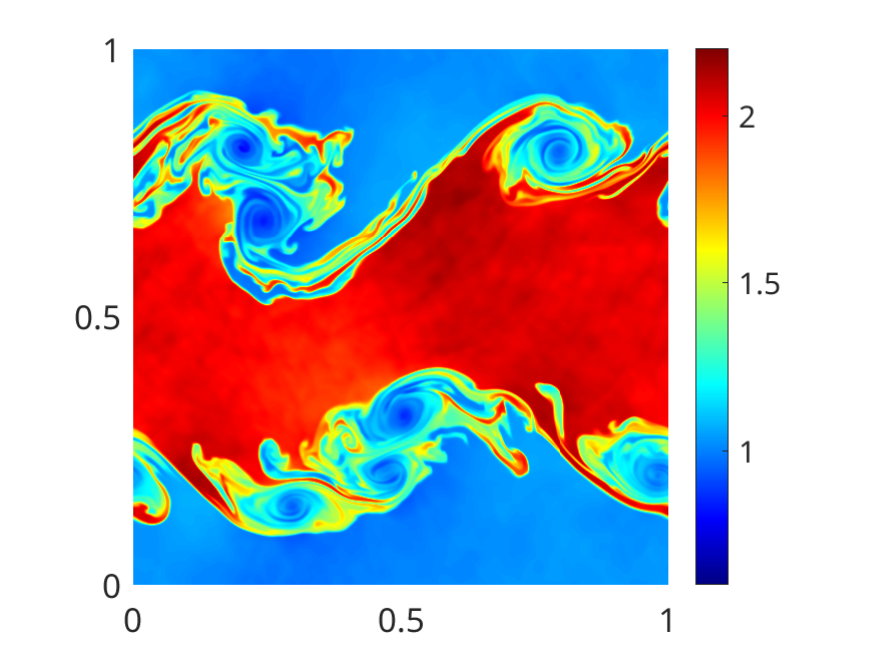}
	}
	\subfigure[third-order VFV method]{
		\includegraphics[width=0.30\linewidth]{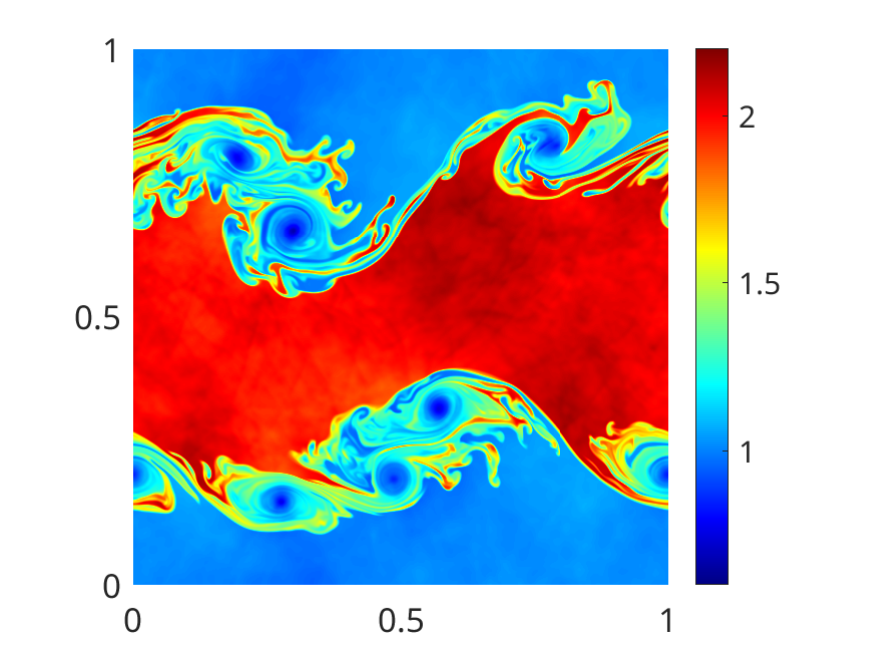}
	}
	\\
	\subfigure[fifth-order VFV method]{
		\includegraphics[width=0.30\linewidth]{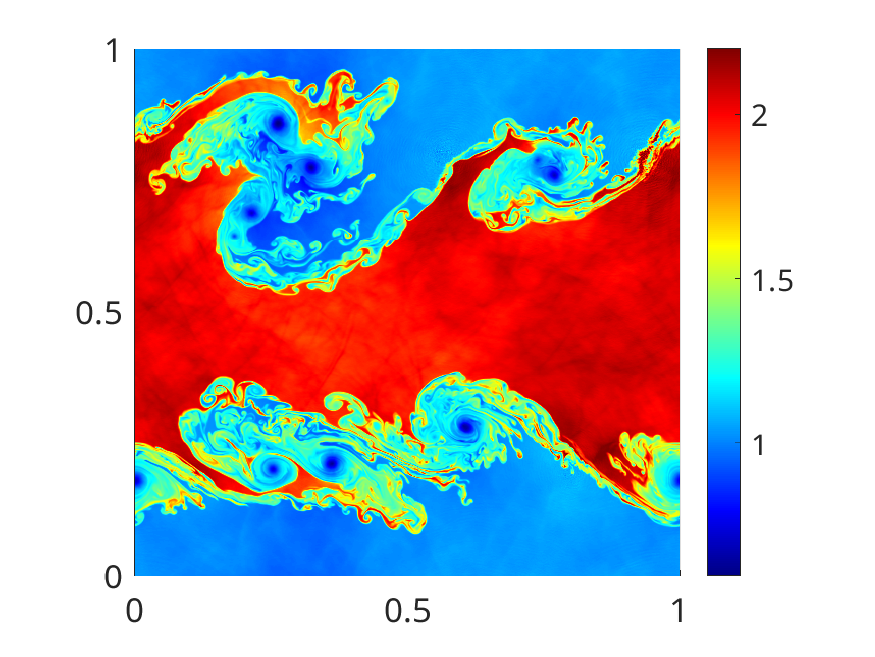}
	}
	\subfigure[seventh-order VFV method]{
		\includegraphics[width=0.30\linewidth]{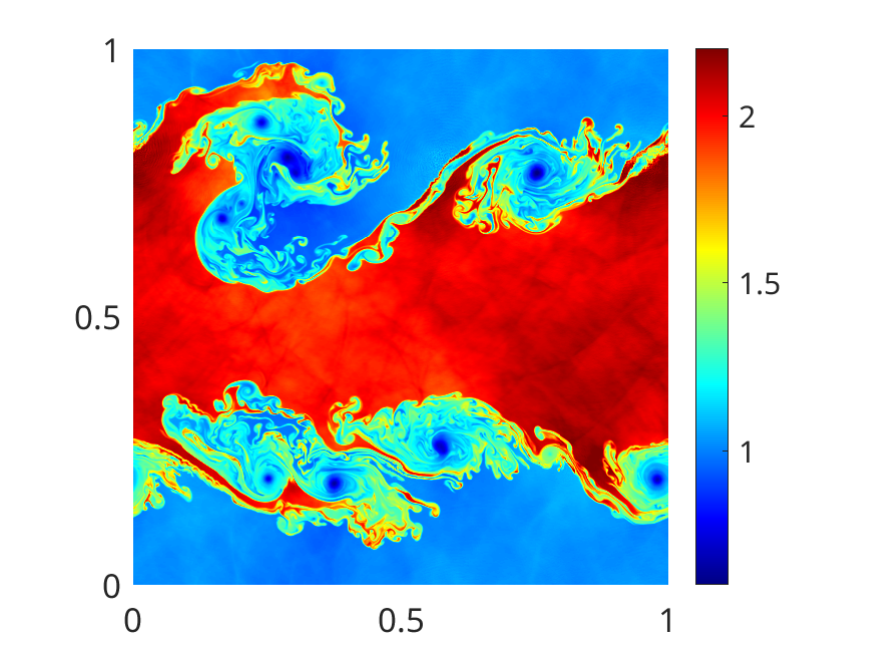}
	}
	\caption{Density computed by  the first and higher-order VFV methods at time $T=2.0$.}
	\label{fig1}
\end{figure}

\begin{figure}[!h]
	\centering
	\subfigure[first-order VFV method]{
		\includegraphics[width=0.30\linewidth]{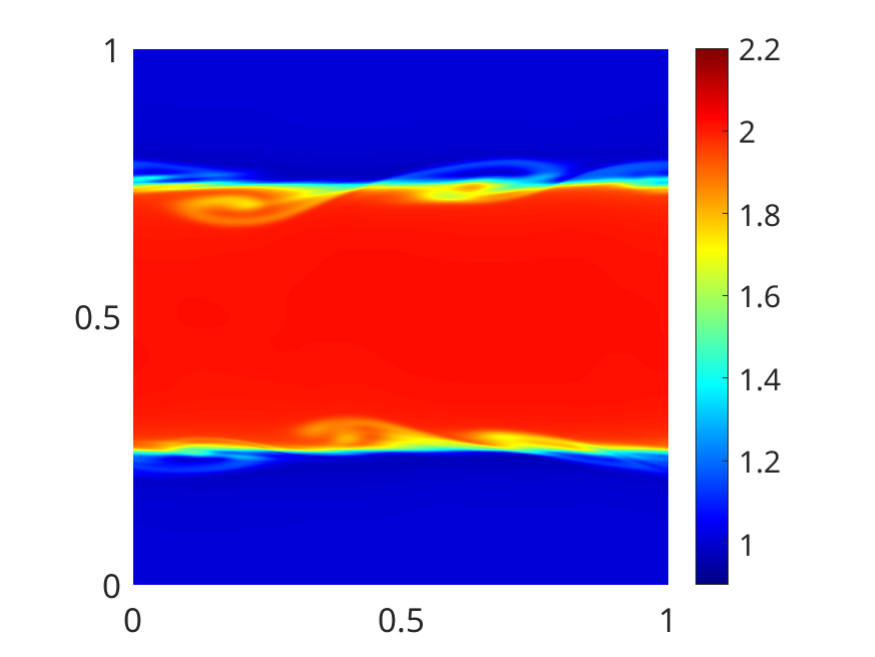}
	}
	\subfigure[second-order VFV method]{
		\includegraphics[width=0.30\linewidth]{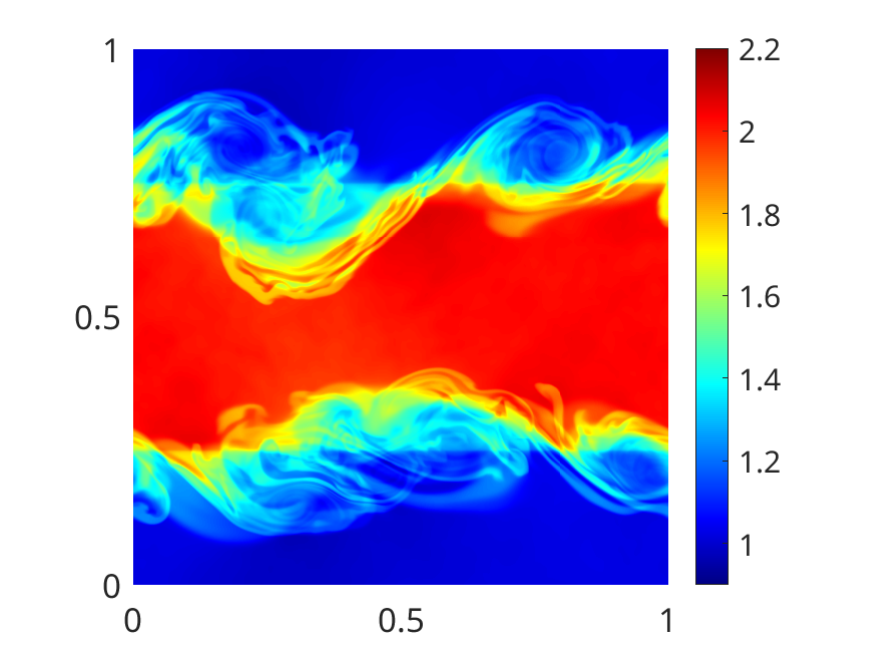}
	}
	\subfigure[third-order VFV method]{
		\includegraphics[width=0.30\linewidth]{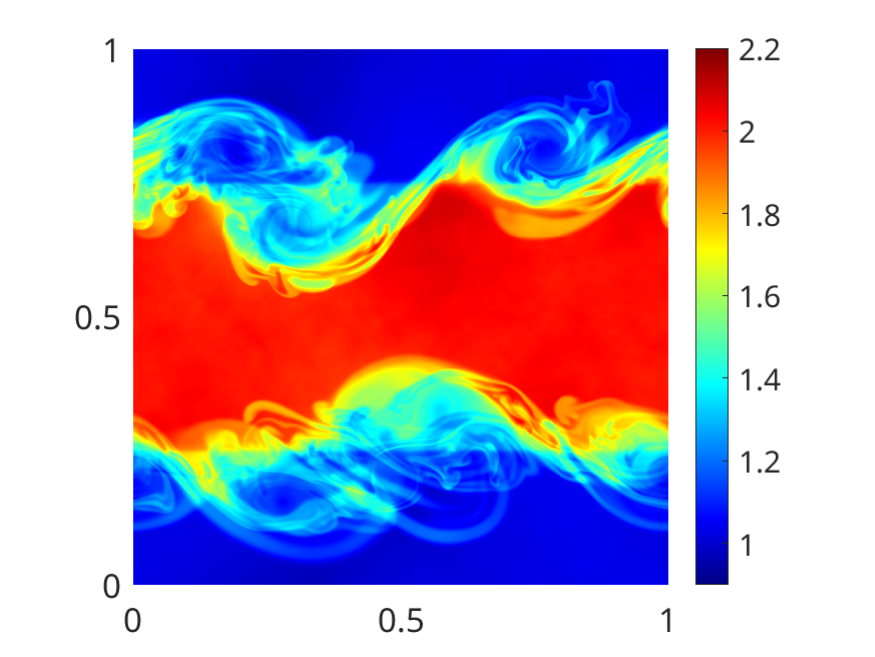}
	}
	\\
	\subfigure[fifth-order VFV method]{
		\includegraphics[width=0.30\linewidth]{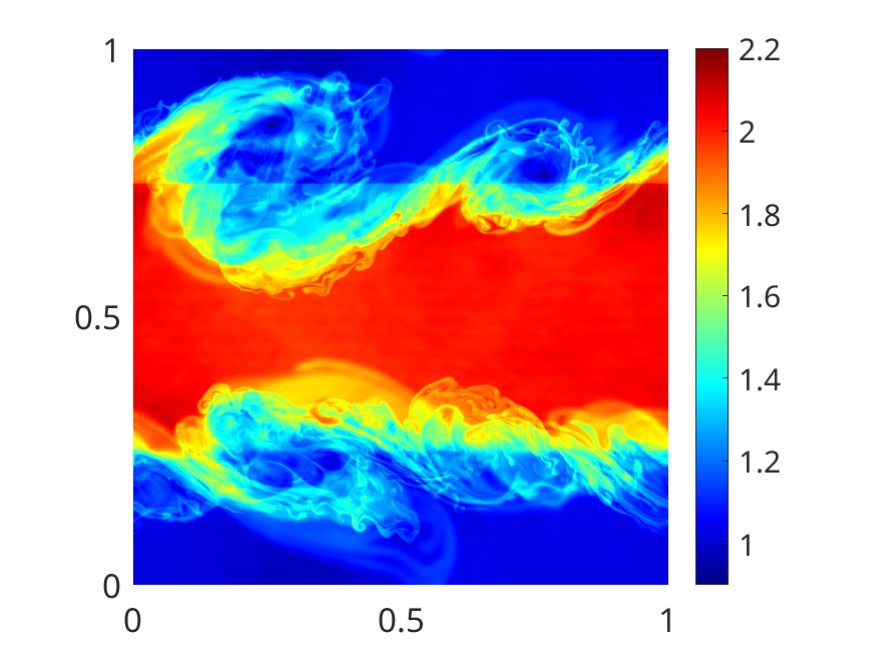}
	}
	\subfigure[seventh-order VFV method]{
		\includegraphics[width=0.30\linewidth]{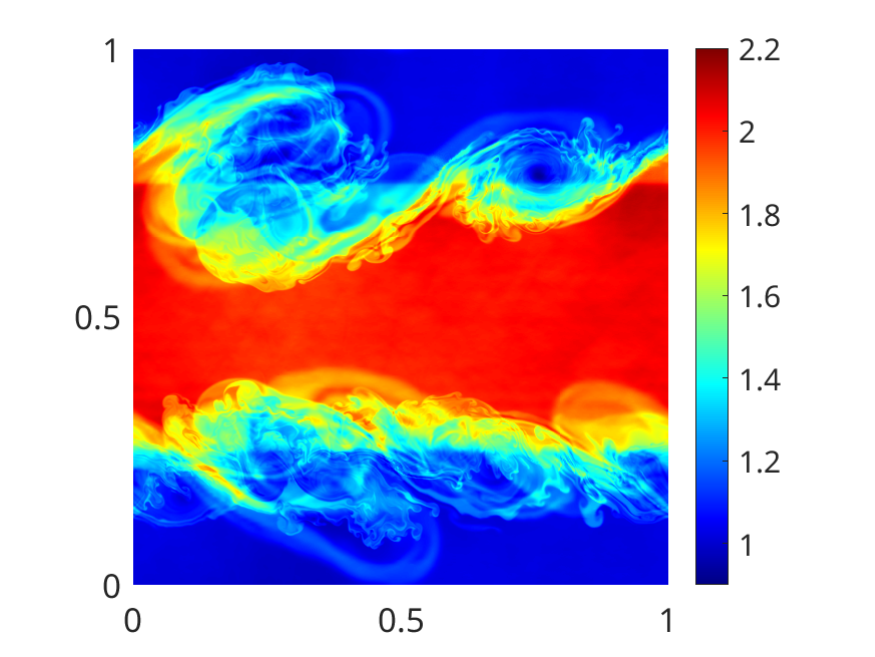}
	}
	\caption{Ces\`aro averages of the density computed by the first and higher-order VFV methods at time $T=2.0$.}
	\label{fig2}
\end{figure}

In this example, initial data lead to a solution of  shear flow-type  with small initial oscillations that are rising up as time progresses. In Figure~\ref{fig1} we present numerical density obtained by the first-order VFV method as well as higher-order VFV methods at time T=2.0. The latter are based on the A-WENO reconstruction and higher-order finite difference flux corrections, see~\cite{Kurganov} for the details of a higher order FV method. 

As a solution remains bounded, the concentration measure $\mathfrak{C} = 0$. We will show below that numerical solutions are oscillatory since 
there is a nonzero oscillation defect, the so-called \emph{Reynolds defect}  
\begin{equation}
	\mathfrak{R} = \bigg \langle \mathcal{V}_{t,x}; \frac{\tvm \otimes \tvm}{\tvr} + p(\tvr,\tilde{S}) \mathbb{I} \bigg \rangle - \frac{\vm \otimes \vm}{\vr} - p(\vr, S) \mathbb{I}
\end{equation}
and nonzero energy defect $D_{\rm E}$, cf.~\eqref{defect}.
Consequently,   $\mathcal{V}_{t,x}  \neq \delta_{[\vr(t,x), \vm(t,x), S(t,x)] }$.

In what follows, we present time evolution of the approximation of the Reynolds defect $\mathfrak{R}$ and the energy
defect $D_{\rm E}$. To this end, we consider a sequence of meshes $\mathcal{T}_{h_n}$ with $h_1=\frac{1}{64},\ h_2=\frac{1}{128},\ h_3=\frac{1}{256},\ h_4=\frac{1}{512},\ h_5=\frac{1}{1024}$ and compute the $L^1$-norm of the Reynolds defect $\mathcal{R}(t)$ as well as the energy defect $D_{\rm E}(t)$ as functions of time.
To shorten the notations we denote the Ces\`aro averages as
\begin{equation}
	\widetilde{\vr}_N= \frac 1 N \sum_{n=1}^{N} \vr_{h_n},\ \widetilde{\vm}_N= \frac 1 N \sum_{n=1}^{N} \vm_{h_n},\ \widetilde{S}_N= \frac 1 N \sum_{n=1}^{N} S_{h_n},\ N=5.
\end{equation}
Figure~\ref{fig2} presents the Ces\`aro averages computed for the first and higher-order VFV methods.  Reconstruction between a coarse and a fine grid is realised by the ninth-order  A-WENO reconstruction, see~\cite{Kurganov}.  We can observe that the Ces\`aro  averages are similar for different VFV methods which is in contrast to Figure~\ref{fig1}. The latter shows that for higher-order methods more oscillatory solutions were obtained.   The following functions of the defects have been computed to quantify the oscillatory behaviour of numerical solutions:
\begin{equation}
	\begin{aligned}
		\mathcal{R}_1(t)&=\norm*{\frac{1}{N} \sum_{n=1}^{N} \bigg[  \frac{\vm_{h_n}(t,\cdot) \otimes \vm_{h_n}(t,\cdot)}{\vr_{h_n}(t,\cdot)}+p(\vr_{h_n}(t,\cdot),S_{h_n}(t\cdot)) \mathbb{I} \bigg] }_{L^1(\Omega)}, \\
		\mathcal{R}_2(t)&=\norm*{\frac{\widetilde{\vm}_N(t,\cdot)\otimes \widetilde{\vm}_N(t,\cdot)}{{\tvr}_N(t,\cdot)}+ p\Big( {\tvr}_N(t,\cdot),\widetilde{S}_N (t,\cdot)\Big) \mathbb{I}}_{L^1(\Omega)}, \\
		\mathcal{R}(t)&=\bigg\lVert \frac{1}{N} \sum_{n=1}^{N} \bigg[  \frac{\vm_{h_n}(t,\cdot) \otimes \vm_{h_n}(t,\cdot)}{\vr_{h_n}(t,\cdot)}+p(\vr_{h_n}(t,\cdot),S_{h_n}(t,\cdot)) \mathbb{I} 
		\bigg] \\
		&\qquad \qquad \qquad \qquad \qquad -  \frac{\widetilde{\vm}_N(t,\cdot)\otimes \widetilde{\vm}_N(t,\cdot)}{{\tvr}_N(t,\cdot)} - p\Big( {\tvr}_N(t,\cdot),\widetilde{S}_N (t,\cdot)\Big) \mathbb{I}  \bigg\lVert_{L^1(\Omega)},\\
		\mathcal{E}_1(t)&= \int_{\Omega} \widetilde{E}_N (t,\cdot) \dx = \int_{\Omega} 
		\frac{1}{N} \sum_{n=1}^{N} \bigg[  \frac{|\vm_{h_n}(t,\cdot)|^2}{2\vr_{h_n}(t,\cdot)}+\vr_{h_n}(t,\cdot)e(\vr_{h_n}(t,\cdot),S_{h_n}(t,\cdot)) \bigg] \dx, \\
		\mathcal{E}_2(t)&= \int_{\Omega} E\Big( \widetilde{\vr}_N(t,\cdot), \widetilde{\vm}_N(t,\cdot), \widetilde{S}_N(t,\cdot) \Big) \\
		&= \int_{\Omega}  \frac{|\widetilde{\vm}_N(t,\cdot)|^2}{2\widetilde{\vr}_N(t,\cdot)}+\widetilde{\vr}_N(t,\cdot) e\Big( \widetilde{\vr}_N(t,\cdot),\widetilde{S}_N(t,\cdot) \Big) \dx,\\
		D_{\rm E}(t) &= \int_{\Omega} \widetilde{E}_N(t,\cdot)-E\Big( \widetilde{\vr}_N(t,\cdot), \widetilde{\vm}_N(t,\cdot), \widetilde{S}_N(t,\cdot) \Big) \dx.
	\end{aligned}
\end{equation}

\begin{figure}
	\centering
	\subfigure[$\mathcal{R}_1(t)$, $\mathcal{R}_2(t)$]{
		\includegraphics[width=0.22\linewidth]{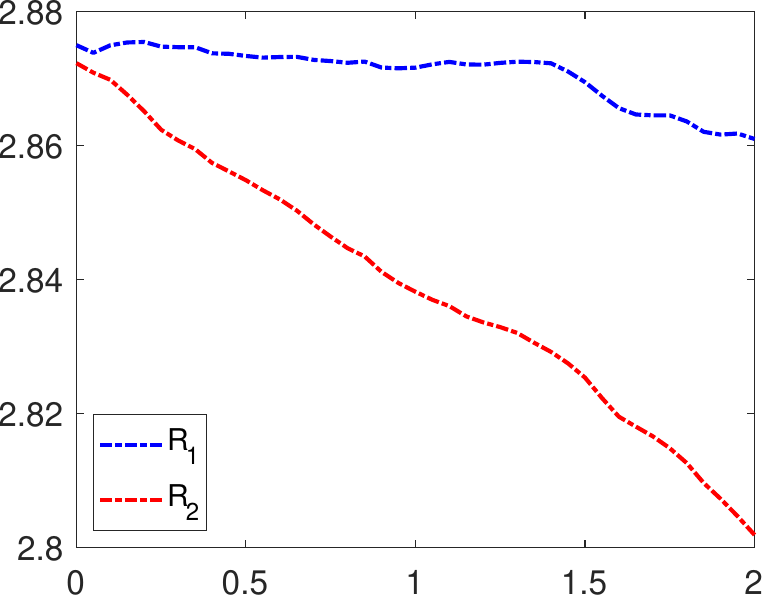}
	}
	\subfigure[$\mathcal{R}(t)$]{
		\includegraphics[width=0.22\linewidth]{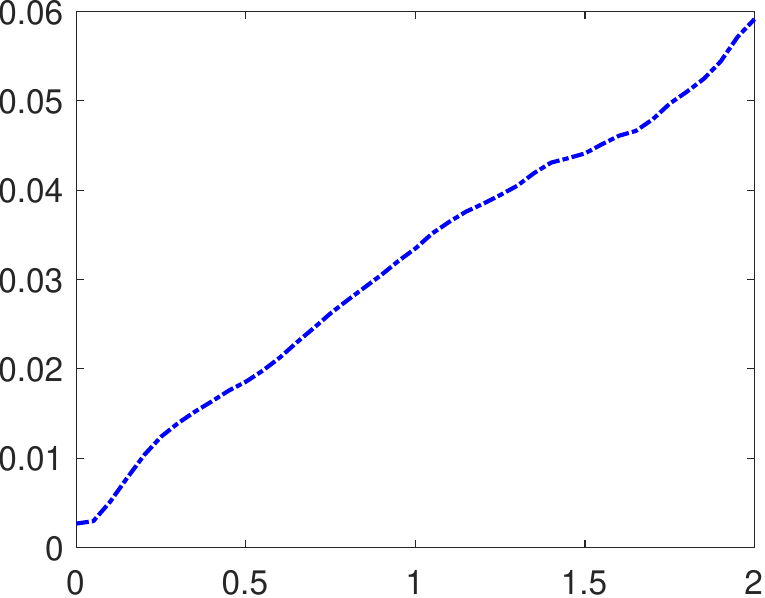}
	}
	\subfigure[$\mathcal{E}_1(t),$ $\mathcal{E}_2(t)$]{
		\includegraphics[width=0.22\linewidth]{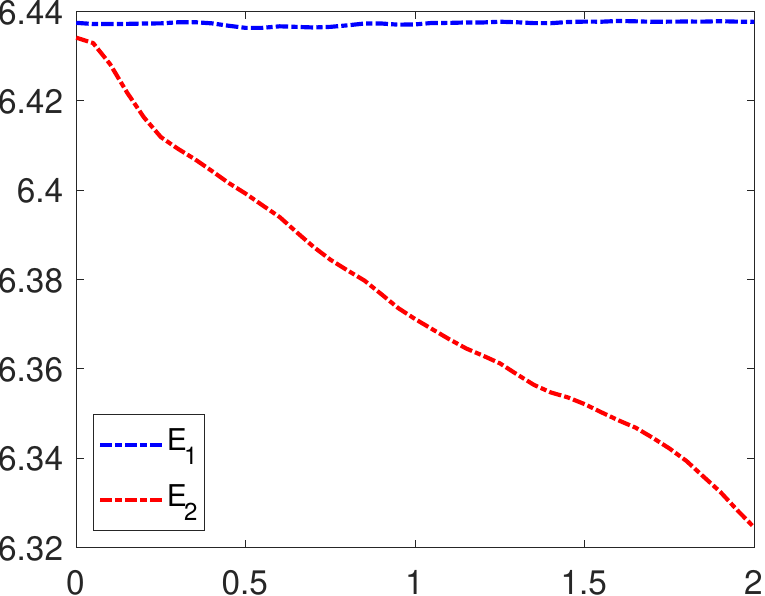}
	}
	\subfigure[$D_E(t)$]{
		\includegraphics[width=0.22\linewidth]{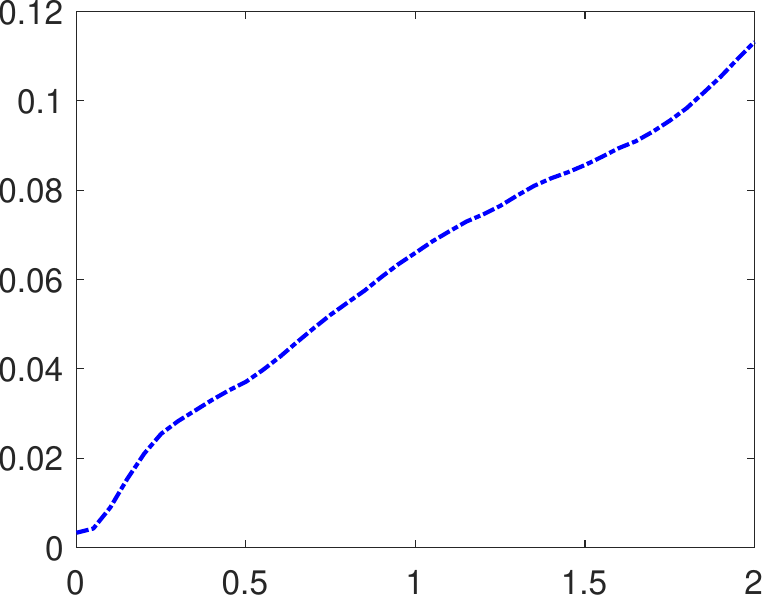}
	}
	\caption{Time evolution of the $L^1$-norm of the Reynolds defect and of the energy defect, $t \in [0,2]$.}	
	\label{fig3a}
\end{figure}

As illustrated in Figure~\ref{fig3a}(c), the Ces\`aro average of the energy is constant in time. However, the energy defect $D_{\rm E}(t) \neq 0$. Indeed, as we observe 
the energy of a dissipative solution satisfies $\int_\Omega E(\vr, \vm, S) \dx<  \int_\Omega \langle \mathcal{V}_{t,x}; E(\tvr, \tvm, \tilde{S})\rangle \dx = \int_\Omega E(\vr_0, \vm_0, S_0) \dx$.
Analogously, we also observe that the $L^1$-norm of the Reynolds defect $\mathcal{R} \neq 0$.  Consequently, according to Theorem~\ref{cadT} the computable DMV solution obtained by the VFV method \eqref{FV},\eqref{flux} is not maximal.  

Our extensive numerical experiments confirm that this is the case for many standard numerical methods for the Euler system, such as DG methods, central-upwind finite volume method, low dissipation central upwind method including their higher-order variants. We refer to our forthcoming paper for further details for different numerical methods \cite{CHKLY}. We conclude this section by presenting the results for the solutions obtained by higher-order VFV method that is based on the higher-order A-WENO  reconstruction and higher-order central difference fluxes, cf.~\cite{Kurganov}.  Table~\ref{tab1} documents that for the first and the higher-order VFV methods numerical solutions are oscillatory, i.e. $D_{\rm E} \neq 0 $. We also  present the total entropy 
\[
\mathcal{S}(t) \equiv   \int_\Omega\widetilde S_N( t,\cdot) \dx, \quad t \in (0,T]
\]
and the entropy defect 
\[
D_{\rm Ent}(t)=  \int_{\Omega}  S\Big( \widetilde{\vr}_N (t,\cdot), \widetilde{\vm}_N (t,\cdot), \widetilde{E}_N (t,\cdot)\Big) - \widetilde{S}_N (t,\cdot) \dx
\]
of the energy-preserving  DMV solutions as described in  Section~\ref{B}. For the latter, the energy equality holds, cf.~\eqref{Bd3a}, and the limiting process to obtain an energy-preserving DMV solution is realized in the conservative variables $(\tvr_N, \tvm_N, \widetilde{E}_N)$ while keeping the entropy as a nonlinear function of the conservative variables. In Figure~\ref{fig4} time evolution of the total entropy $\mathcal{S}(t)$ and  the entropy defect $D_{\rm Ent}(t)$ are presented. Both defects, the energy defect computed for the DMV solution, cf.~Definition~\ref{Dd2},  and the  entropy  defect computed for the energy-preserving DMV solution, cf.~Definition~\ref{B}, are nonzero  increasing functions of time.  Consequently, the numerical solutions do not approximate a maximal computable solution with respect to {$\prec_{\rm D}$}, see Theorem~\ref{cadT}.

\begin{figure}[!h]
	\centering
		\includegraphics[width=0.22\linewidth]{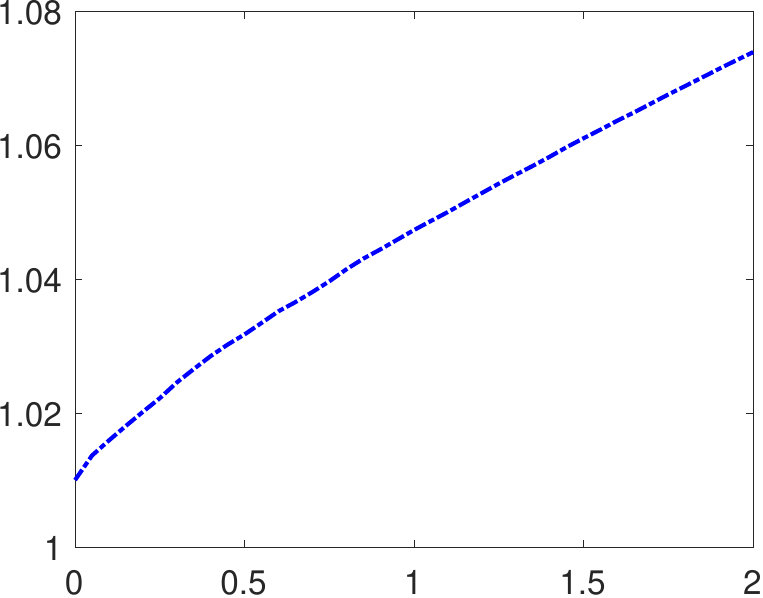}
		\hspace{2cm}
		\includegraphics[width=0.22\linewidth]{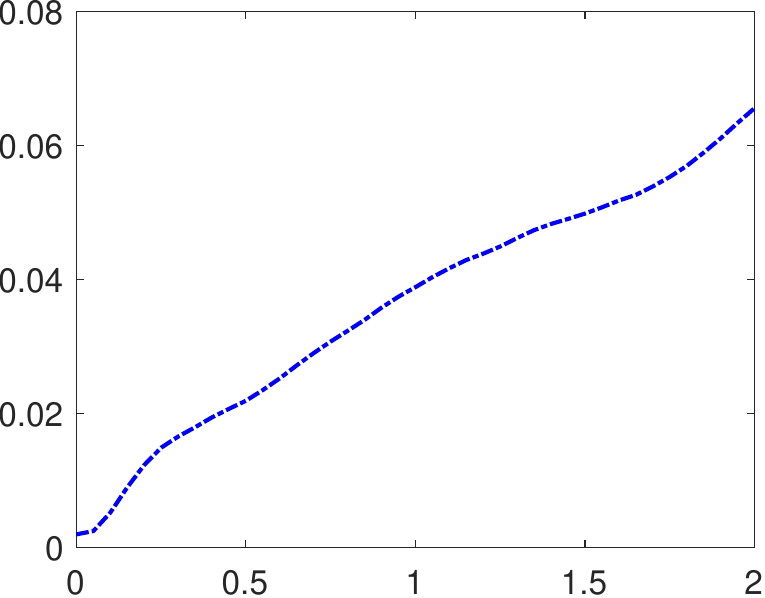}
	\caption{Time evolution of the total  entropy  $\mathcal{S}(t)$ (left) and the  entropy defect $D_{\rm Ent}(t)$(right), $t \in [0,2]$.}
	\label{fig4}
\end{figure}

As non-smooth  solutions of the Euler system are non-unique in general,   we may identify suitable selection criteria. One possibility is to choose a relevant solution by  maximizing the entropy production and  minimizing the defects. Table~\ref{tab1} demonstrates that these selection criteria consistently choose  the same DMV solution: the solution computed by the first-order VFV method, cf.~$\int_0^T \mathcal{S}(t) \dt$, $\int_0^T D_{\rm E}(t) \dt$, and $\int_0^T  D_{\rm Ent} (t) \dt$.

\begin{table}[htbp]
	\centering
	\renewcommand{\arraystretch}{1.5}
	\caption{{Energy defect and entropy production for different VFV methods at $T=2.0$.}}\label{tab1}
	\setlength{\tabcolsep}{1mm}{
		\begin{tabular}{|c|c|c|c|c|c|}
			\hline
			&  first-order & second-order & third-order & fifth-order &seventh-order \\
			\hline
			$\int_0^T \mathcal{E}_1(t) \dt$& 12.8758 &12.8753 &12.8754 &12.8746 &12.8747 \\
			$\int_0^T \mathcal{E}_2(t) \dt$  &12.8495 &12.7825 &12.7745 &12.7515 &12.7508 \\
			$\int_0^T D_{\rm E}(t) \dt$  &0.0264 &0.0929 &0.1009 &0.1232 &0.1239 \\
			$\int_0^T \mathcal{S}(t) \dt$  & 2.1046 &2.0905 &2.0878 &2.0913 &2.0919 \\
			$\int_0^T D_{\rm Ent}(t) \dt$  &  0.0144 & 0.0538 & 0.0587 & 0.0723 & 0.0724 \\
			\hline
	\end{tabular}}
\end{table}

\section{Conclusion}

Oscillatory sequences of consistent approximations obtained by entropy stable
numerical schemes generate DMV solutions of the Euler system. The limit never complies with Dafermos' maximal entropy production criterion. The result is true for both the
conservative-entropy variables $(\vr, \vm, S)$ and the conservative (energy preserving) variables $(\vr, \vm ,E)$. The same numerical scheme, however, may converge (strongly) to a weak solution of the Euler system for certain initial data.

The use of any criterion based on entropy maximization is therefore dubious at least in the class of DMV solutions generated by numerical schemes for certain initial data. A possible explanation of these phenomena may be formulated as the following

\noindent
{\bf Conjecture:}	

\noindent
{\it
	In the case a numerical scheme converges to a weak solution of the Euler system, the limit is a maximal computable DMV solution.
	In the case the same numerical scheme is oscillatory, the maximal (computable) DMV solution does not exist.
}

Note that for a scheme the property of being or not being oscillatory is determined by the choice of initial data. Oscillatory behaviour of some standard numerical schemes was illustrated in Section~\ref{NUM}.


\section*{Acknowledgement} We thank Shaoshuai Chu (Aachen) and Alexander Kurganov (Shenzhen) for providing  us a code for higher order A-WENO reconstruction and higher-order finite difference flux corrections.  M.L.-M. thanks Michael Herty (Aachen) for fruitful discussions on numerical simulations. 

\def\cprime{$'$} \def\ocirc#1{\ifmmode\setbox0=\hbox{$#1$}\dimen0=\ht0
	\advance\dimen0 by1pt\rlap{\hbox to\wd0{\hss\raise\dimen0
			\hbox{\hskip.2em$\scriptscriptstyle\circ$}\hss}}#1\else {\accent"17 #1}\fi}

\end{document}